\documentclass[reqno, 11pt]{amsart}

\usepackage{mathrsfs}
\usepackage{amsbsy,latexsym}
\usepackage{frcursive}
\usepackage{yfonts}
\usepackage{graphicx}
\usepackage{empheq}
\usepackage{mathtools}
\usepackage{tikz}
\usepackage{amsbsy}
\usepackage{avant}
\usepackage{amsmath}
\usepackage{amsthm}
\usepackage{amssymb}
\usepackage{mathrsfs}
\usepackage{amsfonts}
\usepackage{newlfont}
\usepackage[sans]{dsfont}
\usepackage{dcolumn}
\usepackage{bm}
\usepackage{bbm}
\usepackage{stmaryrd}
\usepackage{bbm}
\usepackage{relsize}
\usepackage{euscript}
\usepackage{eufrak}
\usepackage{enumerate}
\usepackage{amsbsy}
\usepackage{fullpage}
\usepackage{mathrsfs}
\usepackage{amsbsy,latexsym}
\usepackage{frcursive}
\usepackage{yfonts}
\usepackage{graphicx}
\usepackage{empheq}
\usepackage{mathtools}
\usepackage{tikz}
\usepackage{amsbsy}
\usepackage{avant}
\usepackage{amsmath}
\usepackage{amsthm}
\usepackage{amssymb}
\usepackage{mathrsfs}
\usepackage{amsfonts}
\usepackage{newlfont}
\usepackage[sans]{dsfont}
\usepackage{dcolumn}
\usepackage{bm}
\usepackage{bbm}
\usepackage{stmaryrd}
\usepackage{bbm}
\usepackage{relsize}
\usepackage{euscript}
\usepackage{eufrak}
\usepackage{enumerate}
\usepackage{amsbsy}
\usepackage{fullpage}

\usepackage[margin=1.0in]{geometry}

\numberwithin{equation}{section}
\newtheorem{thm}{Theorem}[section]

\newtheorem{cor}[thm]{Corollary}
\newtheorem{lem}[thm]{Lemma}
\newtheorem{prop}[thm]{Proposition}
\newtheorem{defn}[thm]{Definition}
\newtheorem{rem}[thm]{Remark}

\begin{document}
\allowdisplaybreaks{
\title[]{THE KELLY CRITERION AND UTILITY FUNCTION OPTIMISATION FOR STOCHASTIC BINARY GAMES:~SUBMARTINGALE AND SUPERMARTINGALE REGIMES}
\author{Steven D Miller}\email{stevendm@ed-alumni.net}
\address{}
\maketitle
\begin{abstract}
A reformulation of the Kelly criterion is presented. Let $\mathfrak{G}$ be a generic stochastic Markovian-Bernoulli binary game with outcomes $\mathscr{Z}(I)\in\lbrace -1,1\rbrace$ of N trials for $I=1...N$. The binomial probabilities are $\bm{\mathsf{P}}(\mathscr{Z}(I)=1)=p$ and $\bm{\mathsf{P}}(\mathscr{Z}(I)=-1)=q$
with $p+q=1$. For a fair game $p=q=\tfrac{1}{2}$ and for a biased game $p>q$. If $\mathscr{W}(0)$ is the initial wealth then at the $I^{th}$ trial one bets a fraction $\mathcal{F}$ so that $B(I)=\mathcal{F}\mathscr{W}(I-1)$. If one wagers $B(I)$ and wins one recovers the original wager plus $B(I)$ if $\mathscr{Z}(I)=+1$, or a loss of $B(I)$ if $\mathscr{Z}(I)=-1$. The wealth at the $N^{th}$ trial/bet for large $N$ is the random walk $\mathscr{W}(N)=\mathscr{W}(0)+\sum_{I=1}^{N}B(I)\mathscr{Z}(I)=\mathscr{W}(0)\prod_{I=1}^{N}(1+\mathcal{F}\mathscr{Z}(I))$ with expectation $\bm{\mathsf{E}}[\mathscr{W}(N)]$. Defining a 'utility function' $\bm{\mathsf{U}}(\mathcal{F},p)=\bm{\mathsf{E}}[\log(\mathscr{W}(N)/\mathscr{W}(0))^{1/N}]$ then $\bm{\mathsf{U}}(\mathcal{F},p)$ is optimised by the Kelly fraction $\mathcal{F}=\mathcal{F}_{K}=p-q=2p-1$, which is essentially a critical point of $\bm{\mathsf{U}}(\mathcal{F},p)$. Also $\bm{\mathsf{U}}(\mathcal{F}_{K},p)$ can be related to the Shannon entropy. If $[0,1]=[0,\mathcal{F}_{*})\bigcup [\mathcal{F}_{*}]\bigcup (\mathcal{F}_{*},1]$ with $\bm{\mathsf{U}}(\mathcal{F}_{*},p))=0$ then $\bm{\mathsf{U}}(\mathcal{F},p))>0, \forall\mathcal{F}\in[0,\mathcal{F}_{*})$ and $\mathscr{W}(N)$ is a submartingale for $p>1/2$; also $\bm{\mathsf{U}}(\mathcal{F},p))<0,\forall \mathcal{F}\in(\mathcal{F}_{*},1]$, and $\mathscr{W}(\mathcal{F},p)$ is a supermartingale. Estimates are derived for variance and volatility $\bm{\mathsf{VAR}}(\mathscr{W}(N))$ and $\bm{\sigma}(\mathscr{W}(N))=\sqrt{\bm{\mathsf{VAR}}(\mathscr{W}(N)})$. For large $N$ and $\mathcal{F}=\mathcal{F}_{K}$, $\bm{\mathsf{E}}[\mathscr{W}(N)]$ grows exponentially.
\end{abstract}
\tableofcontents
\raggedbottom
\maketitle
\clearpage
\section{Introduction}
This autodidactic paper reviews and reformulates the Kelly criterion \textbf{[1,2]} for optimising a utility function within the context of a generic stochastic binary Markovian 'game' $\mathfrak{G}$ of $N$ Bernoulli trials, where one wagers on the outcome of only two possible outcomes which have the probability p and q with $p+q=1$ and $p>q$, and thus an edge $\mathcal{E}=p-q$. All past outcomes have no affect on any future outcomes, hence the game has no 'memory' of the past and is Markovian. The underlying probability distribution is binomial, which by the Central Limit Theorem becomes Gaussian for large samples. The simplest example of this would be a biased coin flipping game\textbf{[3-6]} However, new results are provided in formally identifying the regimes, and  giving proofs, of where the wealth or liquidity grows in the 'Kelly regime' as a submartingale, or is depleted as a supermartingale. In the submartingale regime, classic Doob theorems can also be applied. Estimates are also derived for variance and volatility of the liquidity/wealth accumulation.

A fundamental issue within finance, investing and gambling is to try and find positive expectation opportunities which provide a statistical edge, usually a very small edge, or when the odds (perhaps computed/estimated independently) are in one's favour. However, even when such opportunities can be found a second challenge or problem arises--one must then know how  much of the initial liquidity or capital can be wagered in such a way as to optimise the known player/bettor advantage but also minimise the risk. For example, consider a simple stochastic binary game of a biased coin flip with $p>q$ with p the probability of heads and $q=1-p$ that of tails. If one bets a certain amount and wins then the bet is returned plus a win equal to the bet size. If one loses then one simply loses the bet. If one bets the entire initial liquidity or wealth $\mathscr{W}(0)$ on heads and wins then the wealth after the first bet is $\mathscr{W}(1)=2\mathscr{W}(0)$ so the money is doubled. If one bets all of this again and wins again, then $\mathscr{W}(2)=2\mathscr{W}(1)=4\mathscr{W}(0)$ and so on. The wealth will then grow exponentially (double) as long as one wins. However, as soon as one losses, the entire accumulated wealth/liquidity will be lost. The probability of this is $1-p^{N}$. Since $p>1/2$ but $p<1$ then this tends to 1 as N increases since $p^{N}$ tends to zero. Formally, $\lim_{N\rightarrow\infty}(1-p^{N})=1$. Hence, the ruin probability is absolute, and so ruin is guaranteed if one keeps playing in this way. The problem then is deduce what optimum \emph{fraction} $\mathcal{F}$ of the current wealth to bet on the next trail so that ruin is avoided in the long run; and wealth slowly accumulates, driven  by the law of large numbers. If $\mathscr{W}(N-1)$ is the wealth after the $(N-1)^{th}$ trial/wager then one will bet an amount $B(N)=\mathcal{F}\mathscr{W}(N-1)$ at the $N^{th}$ bet. The key question is: \emph{how then does one find or compute the optimum value of $\mathcal{F}$ for a given advantage $p>q$ with $p+q=1$?}

This problem was properly and rigorously solved by Kelly \textbf{[1,2]}, within the context of information theory, and has been studied quite extensively since \textbf{[7-16]}. However, the problem has been of interest since at least the 18th century in the form of the St Petersburg problem \textbf{[6,17]} and its analysis by Bernoulli. The Kelly criterion essentially aims to find the value of $\mathcal{F}$ for a given advantage, which maximises a "utility function", which is the expectation of the logarithm of wealth. Economists and financial engineers will use jargon-laden terms like "geometric mean maximising portfolio strategy" etc., but they are all essentially referring to a utility function. If $\mathscr{W}(N)$ is the liquidity or wealth accumulated at the $N^{th}$ bet or wager, in a scenario, where there is a advantage or statistical edge $p>q$ , then the utility function is essentially of the form $\bm{\mathsf{E}}[\log\mathscr{W}(N)]$, where $\bm{\mathsf{E}}$ is the expectation.

The concept of utility function as the log of wealth was first introduced by Bernoulli\textbf{[6,17]}. The aim to compute a bet size at each trial, for $I=1...N$ trials, which optimise or maximise this function, and which is a fixed fraction $\mathcal{F}$ of the current wealth; in other words, $\mathcal{F}$ is actually \emph{the critical point for the utility function}. As $\mathscr{W}(N)$ grows one can then risk a larger wager, while at the same time knowing that there is zero (or very little) chance of ruin whereby $\mathscr{W}(N)\rightarrow 0$ for large $N$. Thorp \textbf{[18]} successfully applied the criterion in  selecting bet sizes while card counting in blackjack (which gave him an edge or advantage) in Las Vegas in the 1960s. He later successfully applied what he had learned to securities markets and running a hedge fund \textbf{[8]}. A good reference is \textbf{[10]}. This book is a collection of papers that explore the Kelly criterion in some depth, including its theoretical foundations and practical applications in finance, covering a wide range of topics, from portfolio management to risk control.

In this paper, we consider only the mathematical foundations and derive the Kelly criterion by considering a growth rate or utility function of the form $\bm{\mathsf{U}}
=\bm{\mathsf{E}}[\log(|\mathscr{W}(N)/\mathscr{W}(0)|^{1/N}]$ for a binary Markovian game of $N$ Bernoulli trials. We also derive estimates for variances$\bm{\mathsf{VAR}}(\mathscr{W}(N))$ and volatility $\bm{\sigma}(\mathscr{W}(N))=\sqrt{\bm{\mathsf{VAR}}(\mathscr{W}(N)})$. It will also be shown that for large N, that $\bm{\mathsf{E}}[\mathscr{W}(N)]$ grows exponentially and is a submartingale when $\bm{\mathsf{U}}(\mathcal{F},p)>0.$

\section{Stochastic Binary Games of Bernoulli trials}
The underlying probability distribution of the stochastic game is binomial and the simplest example of such stochastic game would be a coin flip \textbf{[2,3-6]}
\begin{defn}
Let $\mathfrak{G}$ be a stochastic Markovian-Bernoulli binary game (MBBG) of two outcomes $\mathscr{Z}=1$ or $\mathscr{Z}=-1$ or $\mathscr{Z}\in\lbrace -1,1\rbrace$.
A game consists of N independent trials or bets where N is large and one bets on either of the possible outcomes. Each trial or bet
is i.i.d and Markovian so that past outcomes have no effect on the present or future outcomes. The probabilities are
\begin{align}
&\bm{\mathsf{P}}(\mathscr{Z}=+1)={\mathit{p}}\\&
\bm{\mathsf{P}}(\mathscr{Z}=-1)={\mathit{q}}=1-{\mathit{p}}
\end{align}
and $\mathit{p}+\mathit{q}=1$. The game is fair if $\mathit{p}=\mathit{q}=\tfrac{1}{2}$ and biased or advantageous to the player if $\mathit{p}>\mathit{q}$. If $\mathit{p}>\mathit{q}$ and the player bets on the outcome $\mathscr{Z}=+1$ then the game  is biased in the players favour. The underlying probability distribution or "probability mass function" is binomial. Suppose in N trials there are $U$ wins with $\mathscr{Z}=1$ for each win, and $V$ losses with $\mathscr{Z}=-1$ for each loss. Then $N=U+V$. After N trials, the probability that $U$ has a value $U=\alpha$ is given by
\begin{align}
\bm{\mathsf{P}}(U=\alpha;N)=\binom{N}{\alpha}\mathit{p}^{\alpha}q^{N-\alpha}=\binom{N}{\alpha}\mathit{p}^{\alpha}(1-\mathit{p})^{N-\alpha}
\end{align}
Similarly, for $V=\beta$ losses out of N trials
\begin{align}
\bm{\mathsf{P}}(V=\beta;N)=\binom{N}{\beta}\mathit{q}^{\beta}(1-\mathit{q})^{N-\beta}\equiv \binom{N}{\beta}(1-\mathit{p})^{\beta}\mathit{p}^{N-\beta}
\end{align}
\end{defn}
If $N=U+V$ and $U=\alpha$ then one must have $V=\beta$ so that $\alpha+\beta=N$.
One can write
\begin{align}
&U=\alpha=\sim Binom(N,\mathit{p})\nonumber\\&
 V=\beta=\sim Binom (N,\mathit{q})
\end{align}
If $\alpha=1$ and $N=1$, that is one trial, then the probability of one win is
\begin{align}
\bm{\mathsf{P}}(U=1;1)\equiv \bm{\mathsf{P}}(\mathscr{Z}=+1)=\binom{1}{1}\mathit{p}(1-\mathit{p})^{0}=\frac{1!}{1! 0!}\mathit{p}(1-\mathit{p})^{0}=\mathit{p}
\end{align}
Similarly, if $\beta=1,N=1$, then the probability of a loss is
\begin{align}
\bm{\mathsf{P}}(V=1;1)\equiv \bm{\mathsf{P}}(\mathscr{Z}=-1)=\binom{1}{1}\mathit{q}(1-\mathit{q})^{0}=\frac{1!}{1! 0!}\mathit{q}(1-\mathit{q})^{0}=\mathit{q}
\end{align}
Note also that by the binomial series
\begin{align}
&\sum_{\alpha=0}^{N}\bm{\mathsf{P}}(U=\alpha;N)=\sum_{\alpha=0}^{N}\binom{N}{\alpha}p^{\alpha}q^{N-\alpha}=(p+q)^{N}=1
\\&
\sum_{\beta=0}^{N}\bm{\mathsf{P}}(V=\beta;N)=1
\end{align}
since $p=q=1$.

The binary game is also Markovian in that past outcomes have no affect on the present outcome or any future outcomes. This means that the transition probabilities are
\begin{align}
&\bm{\mathsf{P}}(U=\alpha;N|U=\alpha-1;N-1)=\mathit{p}~if~\alpha=(\alpha-1)+1\nonumber\\&
\bm{\mathsf{P}}(U=\alpha;N|U=\alpha-1;N-1)=\mathit{q}=1-\mathit{p}~if~\alpha=\alpha-1\nonumber\\&
\bm{\mathsf{P}}(V=\beta;N|V=\beta-1;N-1)=\mathit{q}~if~\beta=(\beta-1)+1\nonumber\\&
\bm{\mathsf{P}}(V=\beta;N|V=\beta-1;N-1)=\mathit{q}=1-\mathit{p}~if~\beta=\beta-1
\end{align}
\begin{defn}
If $\mathit{p}>\mathit{q}$ (e.g, a biased coin) then the game has a positive edge $\mathcal{E}$ given by
\begin{align}
\mathcal{E}=\mathit{p}-\mathit{q}>0
\end{align}
For example, if $\mathit{p}=0.52$ and $\mathit{q}=0.48$ then $\mathcal{E}=0.04$. But for a fair game $\mathit{p}=\mathit{q}=\frac{1}{2}$ and $\mathsf{E}=0$.
\end{defn}
If there are N trials then they can be labelled as $I=1...N$ so that $\bm{\mathsf{P}}(\mathscr{Z}(I)=+1)=P$ and $\bm{\mathsf{P}}(\mathscr{Z}(I)=-1)=Q$ for all $I\in[1,N]$. The initial liquidity is $\mathscr{W}(0)$ and at each trial I, a fraction $\mathcal{F}$ of the current liquidity is wagered so that the 'bet liquidity' is
\begin{align}
B(I)=\mathcal{F}\mathscr{W}(I-1)
\end{align}
Note that after the $(I-1)^{th}$ trial, the current liquidity $\mathscr{W}(I-1)$ is no longer a stochastic quantity and is now known with certainty. Suppose $p<q$ and player always wagers on the favourable bet with probability p. If $B(I)$ is wagered then if $\mathscr{Z}(I)=1$ the player wins $B(I)$ with probability p, getting back the bet $B(I)$ plus another $B(I)$ or $2B (I)$. If $\mathscr{Z}(I)=-1$ then the player loses an amount $B(I)$ with probability q. Then the liquidity $\mathscr{W}(I-1)$ is increased or decreased such  that
\begin{align}
&\mathscr{W}(I)=\mathscr{W}(I-1)+B(I)\\&
\mathscr{W}(I)=\mathscr{W}(I-1)-B(I)
\end{align}
This can then be represented as a 1-dimensional random walk so that for $I=1$ to N trials or wagers, the liquidity $\mathscr{W}(N)$ is then
\begin{align}
\mathscr{W}(N)=\mathscr{W}(0)+\sum_{I=1}^{N}\mathscr{Z}(I)B(I)
\end{align}
with expectation
\begin{align}
\bm{\mathsf{E}}[\mathscr{W}(N)]=\mathscr{W}(0)+\bm{\mathsf{E}}\left[\sum_{I=1}^{N}\mathscr{Z}(I)B(I)\right]
\end{align}
\begin{lem}
The expected values, variances and volatilities of U and V are the standard results
\begin{align}
&\bm{\mathsf{E}}[U]=N\mathit{p}\\&
\bm{\mathsf{E}}[V]=N\mathit{q}=N(1-\mathit{p})\\&
\bm{\mathsf{VAR}}[U]=N\mathrm{pq}=N\mathit{p}(1-\mathit{p})\\&
\bm{\mathsf{VAR}}=N\mathrm{pq}=N(\mathit{q}(1-\mathit{q})\\&
\bm{\sigma}[U]=\sqrt{N\mathrm{pq}}=\sqrt{N\mathit{p}(1-\mathit{p})}\\&
\bm{\sigma}[V]=\sqrt{N\mathrm{pq}}=\sqrt{N(\mathit{q}(1-\mathit{q})}
\end{align}
\end{lem}
\begin{proof}
We will prove (2.17) and (2.18) only as these are standard results for a binomial distribution. To prove that $p=\bm{\mathsf{E}}[\tfrac{U}{N}]$
\begin{align}
&\bm{\mathsf{E}}[U]=\sum_{\alpha=0}^{N}\alpha\bm{\mathsf{P}}(U=\alpha)=\sum_{\alpha=0}^{N}\alpha \binom{N}{\alpha}\mathit{p}^{\alpha}\mathit{q}^{N-\alpha}\nonumber\\&
=\sum_{\alpha=0}^{N}N\binom{N-1}{\alpha-1}\mathit{p}^{\alpha}\mathit{q}^{N-\alpha}=N\sum_{\alpha=0}^{N}\binom{N-1}{\alpha-1}\mathit{p}^{\alpha}\mathit{q}^{N-\alpha}
\end{align}
since
\begin{align}
\alpha\binom{N}{\alpha}=\binom{N-1}{\alpha-1}
\end{align}
Let $\zeta=\alpha-1$ so that $\alpha=\zeta+1$ and $0\le \zeta\le N-1$ giving
\begin{align}
\bm{\mathsf{E}}[U]=N\sum_{j=0}^{N-1}\binom{N-1}{\zeta}\mathit{p}^{\zeta+1}\mathit{q}^{(N-1)-j}=N\mathit{p}\sum_{j=0}^{N-1}\binom{N-1}{\zeta}\mathit{p}^{\zeta}\mathit{q}^{(N-1)-j}
\end{align}
Then
\begin{align}
\bm{\mathsf{E}}\left[\frac{U}{N}\right]=\frac{1}{N}\bm{\mathsf{E}}[U]=\frac{N\mathit{p}
\sum_{j=0}^{N-1}\binom{N-1}{\zeta}\mathit{p}^{\zeta}\mathit{q}^{(N-1)-j}}{N}=\mathit{p}\sum_{j=0}^{N-1}\binom{N-1}{\zeta}\mathit{p}^{\zeta}\mathit{q}^{(N-1)-j}
\end{align}
One now recognises the binomial expansion
\begin{align}
\sum_{j=0}^{N-1}\binom{N-1}{\zeta}\mathit{p}^{\zeta}\mathit{q}^{(N-1)-j}=(\mathit{p}+\mathit{q})^{N-1}=1
\end{align}
since $p+q=1$. Hence $p=\bm{\mathsf{E}}[U/N]$. It follows that $q=\bm{\mathsf{E}}[V/N]$ by a similar argument.
\end{proof}
\begin{lem}
Since $U$ and $V$ are independent, the covariance is zero.
\begin{align}
\bm{\mathsf{COV}}(U,V)=0
\end{align}
\end{lem}
\begin{proof}
Since $U$ and $V$ represent the outcomes of two independent sets of Bernoulli trials, such as coin tosses, their outcomes are independent of each other; for example, the result of one set of coin tosses does not affect the result of the other set. The game is Markovian with no memory so the outcome of one trial or toss, or any previous trials, does not affect the outcome of any future trials or tosses. The covariance is then
\begin{align}
\bm{\mathsf{COV}}(U,V)=\bm{\mathsf{E}}[UV]-\bm{\mathsf{E}}[U]\bm{\mathsf{E}}[V]=\bm{\mathsf{E}}[U]\bm{\mathsf{E}}[V]-\bm{\mathsf{E}}[U]\bm{\mathsf{E}}[V]=0
\end{align}
\end{proof}
\begin{lem}
The net number of wins for a large sample N will be $W=U-V$ where $N=U+V$. The variance of $W$ is then
\begin{align}
\bm{\mathsf{VAR}}(W)=\bm{\mathsf{VAR}}(U-V)=2Np(1-p)
\end{align}
and the volatility is
\begin{align}
\bm{\sigma}(W)=\sqrt{\bm{\mathsf{VAR}}(W)}=\sqrt{2Np(1-p)}
\end{align}
\end{lem}
\begin{proof}
Given two random variables $\mathscr{X},\mathscr{Y}$, the variance is defined as
\begin{align}
\bm{\mathsf{VAR}}(\mathscr{X}-\mathscr{Y})=\bm{\mathsf{VAR}}(\mathscr{X})+\bm{\mathsf{VAR}}(\mathscr{Y})-2\bm{\mathsf{COV}}(\mathscr{X},\mathscr{Y})
\end{align}
so
\begin{align}
\bm{\mathsf{VAR}}(W)&=\bm{\mathsf{VAR}}({U}-{V})=\bm{\mathsf{VAR}}(U)+\bm{\mathsf{VAR}}(V)-2\bm{\mathsf{COV}}(U,V))\nonumber\\&
=\bm{\mathsf{VAR}}(U)+\bm{\mathsf{VAR}}(V)=Np(1-p)+Np(1-p)=2Np(1-p)
\end{align}
\end{proof}
The variance decreases as p increases above $p=1/2$ and is zero at $p=1$ since there is now higher probability of a favourable outcome
\subsection{Shannon informational entropy}
The Shannon informational entropy can be applied to the binary game and gives a measure of randomness or uncertainty; indeed, the Kelly criterion was originally deduced within the formalism of information theory. When there is complete information or zero uncertainty, the Shannon entropy is zero and the scenario is deterministic. The extremum or maximum of the Shannon entropy also corresponds to maximum randomness and zero information \textbf{[18]}.
\begin{defn}
If $\mathscr{Y}$ is a discrete random variable with probability distribution $\bm{\mathsf{P}}(Y=y)$, where y is the possible outcome of the random variable $\mathscr{Y}$. The Shannon entropy is then a sum over all possible outcomes
\begin{align}
\mathscr{H}(\mathscr{Y})=-\sum_{x}\bm{\mathsf{P}}(\mathscr{Y}=y)\log \bm{\mathsf{P}}(\mathscr{Y}=y)
\end{align}
where log is the natural log. One can also use the logarithm to base 2 if one wishes to utilise "bits" of information. Applying this to the binary game $\mathfrak{G}$ and the random variable $a\mathscr{Z}$ then
\begin{align}
\mathscr{H}(p,q)&=-\sum_{y}\bm{\mathsf{P}}(\mathscr{Z}=y)\log\bm{\mathsf{P}}(\mathscr{Z}=y)\nonumber\\&
=\bm{\mathsf{P}}(\mathscr{Z}=+1)\log\bm{\mathsf{P}}(\mathscr{Z}=+1)-\bm{\mathsf{P}}(\mathscr{Z}=-1)\log\bm{\mathsf{P}}(\mathscr{Z}=-1)\nonumber\\&
=-p\log p-q log q=-p\log p-(1-p)\log (1-p)
\end{align}
\end{defn}
\begin{lem}
The Shannon entropy $\mathscr{H}(p,q)$ is a concave function maximised for $p=q=\frac{1}{2}$ so that
\begin{align}
\mathscr{H}\left(\frac{1}{2},\frac{1}{2}\right)=\log (2)
\end{align}
\end{lem}
\begin{proof}
The informational entropy is
\begin{align}
\mathscr{H}({p},{q})=-{p}\log({p})-{q}\log({q})=-{p}\log({p})-{1-p}\log({1-p})
\end{align}
so that
\begin{align}
\mathscr{H}\left(\frac{1}{2},\frac{1}{2}\right)&=-\frac{1}{2}\log\left(\frac{1}{2}\right)-\frac{1}{2}\log\left(\frac{1}{2}\right)\nonumber\\&
=-\frac{1}{2}(\log 1 -\log 2)-\frac{1}{2}(\log 1-\log 2)=\frac{1}{2}\log(2)+\frac{1}{2}\log(2)=\log(2)
\end{align}
Also the maximum occurs for when
\begin{align}
&\frac{d\mathscr{H}({p})}{d{P}}=\frac{d}{d{p}}[-{p}\log({p})-(1-{p})\log(1-{p})\nonumber\\&
=-[\log({p})]+[\log(1-{p})+1]=\log({p})+\log(1-{p})= \log\left(\frac{1-{p}}{{p}}\right)=0
\end{align}
so that $(1-{p})/{p}=\exp(0)=1$ giving $2{p}=1$ or ${p}=1/2$. For ${p}=1$ and ${q}=0$, for example a double-headed coin, the entropy is $\mathscr{H}(1)=0$. Zero informational entropy implies that the game is totally predictable with no uncertainty.
\end{proof}
\begin{lem}
The binomial probabilities $\bm{\mathsf{P}}(U=\alpha)$ and $\bm{\mathsf{P}}(V=\beta)$ have the entropies $\mathscr{H}(U)=\alpha)$ and $\mathscr{H}(V=\beta)$ given by
\begin{align}
\mathscr{H}(U=\alpha)&=-\sum_{\alpha=0}^{N}\binom{N}{\alpha}\mathit{p}^{\alpha}(1-\mathit{p})^{N-\alpha}\log\binom{N}{\alpha}
-\sum_{\alpha=0}^{N}\binom{N}{\alpha}\mathit{p}^{\alpha}(1-\mathit{p})^{N-\alpha}\alpha \log p \nonumber\\&
-\sum_{\alpha=0}^{N}\binom{N}{\alpha}\mathit{p}^{\alpha}(1-\mathit{p})^{N-\alpha}\alpha \log (1-\mathit{p})
\end{align}
and
\begin{align}
\mathscr{H}(V=\beta)&=-\sum_{\beta=0}^{N}\binom{N}{\beta}\mathit{p}^{\beta}(1-\mathit{q})^{N-\beta}\log\binom{N}{\beta}
-\sum_{\beta=0}^{N}\binom{N}{\beta}\mathit{q}^{\beta}(1-\mathit{q})^{N-\beta}\beta \log q \nonumber\\&
-\sum_{\beta=0}^{N}\binom{N}{\beta}\mathit{q}^{\beta}(1-\mathit{q})^{N-\beta}\beta \log (1-\mathit{q})
\end{align}
\end{lem}
\begin{proof}
The probability of $U=\alpha$ wins in N trials is
\begin{align}
\bm{\mathsf{P}}(U=\alpha)=\binom{N}{\alpha}\mathit{p}^{\alpha}(1-\mathit{p})^{N-\alpha}
\end{align}
so the entropy is
\begin{align}
&\mathscr{H}(U=\alpha)=-\sum_{\alpha=0}^{N}\bm{\mathsf{P}}(U=\alpha)\log \bm{\mathsf{P}}(U=\alpha)\nonumber\\&
=-\sum_{\alpha=0}^{N}\binom{N}{\alpha}\mathit{p}^{\alpha}(1-\mathit{p})^{N-\alpha}\log\binom{N}{\alpha}\mathit{p}^{\alpha}(1-\mathit{p})^{N-\alpha}\nonumber\\&
=-\sum_{\alpha=0}^{N}\binom{N}{\alpha}\mathit{p}^{\alpha}(1-\mathit{p})^{N-\alpha}\log\binom{N}{\alpha}\nonumber\\&
-\sum_{\alpha=0}^{N}\binom{N}{\alpha}\mathit{p}^{\alpha}(1-\mathit{p})^{N-\alpha}\alpha \log\mathit{p}\nonumber\\&
-\sum_{\alpha=0}^{N}\binom{N}{\alpha}\mathit{p}^{\alpha}(1-\mathit{p})^{N-\alpha}\alpha \log (1-\mathit{p})
\end{align}
and similarly for $\mathscr{H}(V=\beta)$.
\end{proof}
\begin{cor}
If $p=1,q=0$ then the game is deterministic and one always wins every trial/outcome by always betting on $\mathscr{Z}=1$. The entropies are then zero so that
$\mathscr{H}(U=\alpha)=\mathscr{H}(V=\beta)=\beta$
\end{cor}
\begin{cor}
If $N=1$ and $\alpha=1$ then $\mathscr{H}(U=1)=-p\log p-q\log q $, the entropy for a single trial or outcome.
\end{cor}
\begin{proof}
\begin{align}
&\mathscr{H}(U=\alpha=1)=-\sum_{\alpha=0}^{1}\binom{1}{\alpha}\mathit{p}^{\alpha}(1-\mathit{p})^{1-\alpha}\log\left[\binom{1}{\alpha}\mathit{p}^{\alpha}(1-\mathit{p})^{N-\alpha}\right]\nonumber\\&
=-\binom{1}{0}\mathit{p}^{0}(1-\mathit{p})^{1}\log\left[\binom{1}{0}\mathit{p}^{0}(1-\mathit{p})^{1}\right]-\binom{1}{1}\mathit{p}^{1}(1-\mathit{p})^{0}\log\left[\binom{1}{1}\mathit{p}^{1}(1-\mathit{p})^{0}\right]\nonumber\\&
=-(1-\mathit{p})\log (1-\mathit{p})-\mathit{p}\log \mathit{p}\equiv -\mathit{p}\log \mathit{p}-\mathit{q}\log \mathit{q}
\end{align}
\end{proof}
\section{Utility function optimisation and derivation of the Kelly criterion}
If $\mathscr{W}(0)$ is the initial liquidity for a stochastic binary Markovian game $\mathfrak{G}$ of Bernoulli trials with $\mathscr{Z}(I)\in\lbrace-1,1\rbrace$ for $I=1...N$ then $\bm{\mathsf{P}}(\mathscr{Z}(I)=+1)={p}$ and $\bm{\mathsf{P}}(\mathscr{Z}(I)=-1)={q}$ with ${p}>{q}$ and ${p}+{q}=1$. The growth of the liquidity can be interpreted as a random walk
\begin{align}
\mathscr{W}(N)=\mathscr{W}(0)+\sum_{I=1}^{N}B(I)\mathscr{Z}(I)
\end{align}
with expectation
\begin{align}
\bm{\mathsf{E}}[\mathscr{W}(N)]=\mathscr{W}(0)+\sum_{I=1}^{N}\bm{\mathsf{E}}[B(I)\mathscr{Z}(I)]
\end{align}
If ${p}>\mathit{q}$ then for $N\gg M$ one will have $\bm{\mathsf{E}}[\mathscr{W}(N)]>\bm{\mathsf{E}}[\mathscr{W}(M)]$ so the process is essentially a submartingale. (This is discussed in more detail in Section 4.) The Kelly criterion is essentially the optimal fraction ${\mathcal{F}}_{(K)}$ of the current liquidity that should be wagered when there is an edge with ${p}>{q}$ for a stochastic binary Markovian game. The optimum bet size is then
\begin{align}
B(I)=\mathcal{F}_{K}\mathscr{W}(I-1)
\end{align}
for $I=1...N$. Note again that $B(I)$ is deterministic since after the $(I-1)^{th}$ bet, the current wealth $\mathscr{W}(I-1)$ is now known with complete certainty and is non stochastic. Then
\begin{align}
\bm{\mathsf{E}}[\mathscr{W}(N)]=\mathscr{W}(0)+\sum_{I=1}^{N}\bm{\mathsf{E}}[\mathcal{F}_{K}\mathscr{W}(I-1)\mathscr{Z}(I)]
\end{align}
This bet liquidity choice maximises or optimises the growth rate or growth rate coefficient and minimises the risk. If $\mathscr{W}(N)$ is the liquidity after the $N^{th}$ bet then in the limit as $N\rightarrow\infty$ one has the probabilities
\begin{align}
&\lim_{N\rightarrow\infty}[\bm{\mathsf{P}}(\mathscr{W}(N)\rightarrow\infty)]=1\\&
\lim_{N\rightarrow\infty}[\bm{\mathsf{P}}(\mathscr{W}(N)\rightarrow 0)]=0
\end{align}
\begin{lem}
Let $\mathscr{W}(0)$ be the initial liquidity or bankroll and let $\mathscr{W}(N)$ be the liquidity at the $N^{th}$ trial or bet. At each trial/bet a fraction $\mathcal{F}$ of the current liquidity is wagered; that is, if the liquidity is $\mathscr{W}(N-1)$ then the bet liquidity or bet size at the $N^{th}$ trial is $\mathrm{B}(N)=\mathcal{F}\mathscr{W}(N-1)$. Let $N=U+V$ where $U$ is the number of wins and $V$ is the number of losses. Then
\begin{align}
\mathscr{W}(N)=\mathscr{W}(0)(1+\mathcal{F})^{U}(1-\mathcal{F})^{V}
\end{align}
\end{lem}
\begin{proof}
If $\mathscr{W}(0)$ is the initial liquidity then for each trial a win increases the liquidity by a factor $(1+\mathcal{F})$ or reduces it by a factor $(1-\mathcal{F})$. Hence
\begin{enumerate}
\item If $\mathscr{Z}=+1$ then
\begin{align}
\mathscr{W}(N)=\mathscr{W}(N-1)+\mathcal{F}\mathscr{W}(N-1)=\mathscr{W}(N-1)(1+\mathcal{F})
\end{align}
\item If $\mathscr{Z}=-1$ then
\begin{align}
\mathscr{W}(N)=\mathscr{W}(N-1)-\mathcal{F}\mathscr{W}(N-1)=\mathscr{W}(N-1)(1-\mathcal{F})
\end{align}
\end{enumerate}
This is equivalent to
\begin{align}
\mathscr{W}(N)=\mathscr{W}(0)\prod_{I=1}^{N}(1+\mathcal{F}\mathscr{Z}(I))
\end{align}
where $\mathscr{Z}(I)=1$ for a win and $\mathscr{Z}(I)=-1$ for a loss. Then after U wins and V losses
\begin{align}
&\mathscr{W}(U)=\mathscr{W}(0)\prod^{U}(1+\mathcal{F})=\mathscr{W}(0)(1+\mathcal{F})^{U}\\&
\mathscr{W}(V)=\mathscr{W}(0)\prod^{V}(1-\mathcal{F})=\mathscr{W}(0)(1-\mathcal{F})^{V}
\end{align}
The liquidity at the $N^{th}$ trial is then a product of U wins and V losses.
\begin{align}
\mathscr{W}(N)=\mathscr{W}(0)(1+\mathcal{F})^{U}(1-\mathcal{F})^{V}
\end{align}
with $N=U+V$.
\end{proof}
\begin{defn}
The utility function or growth factor coefficient $\bm{\mathsf{U}}(\mathcal{F},p)$ is defined as
\begin{align}
{\bm{\mathsf{U}}}(\mathcal{F},p)=\bm{\mathsf{E}}\left[\log \left|\frac{\mathscr{W}(N)}{\mathscr{W}(0)}\right|^{1/N}\right]=\bm{\mathsf{E}}[\log|\mathscr{G}(\mathcal{F})|^{1/N}]
\end{align}
Then from (3.7) it follows that
\begin{align}
\bm{\mathsf{U}}(\mathcal{F},p)=\bm{\mathsf{E}}\left[\log\lbrace(1+\mathcal{F})^{U/N}(1-\mathcal{F})^{V/N}\rbrace\right]
\end{align}
\end{defn}
\begin{lem}
The GFC can be written as
\begin{align}
\bm{\mathsf{U}}(\mathcal{F},p)=\mathit{p}\log(1+\mathcal{F})+\mathit{q}\log(1-\mathcal{F})
\end{align}
\end{lem}
\begin{proof}
\begin{align}
&\bm{\mathsf{U}}(\mathcal{F},p)=\bm{\mathsf{E}}[\log(1+\mathcal{F})^{U/N}(1-\mathcal{F})^{V/N}]\nonumber\\&=\bm{\mathsf{E}}[\log(1+\mathcal{F})^{U/N}+\log(1-\mathcal{F})^{V/N}]\nonumber\\&
=\bm{\mathsf{E}}\left[\frac{U}{N}\log(1+\mathcal{F})+\frac{V}{N}\log(1-\mathcal{F})\right]\nonumber\\&=\bm{\mathsf{E}}\left[\frac{U}{N}\right]\log(1+\mathcal{F})+\bm{\mathsf{E}}\left[\frac{V}{N}\right]\log(1-\mathcal{F})\nonumber\\&
\nonumber\\&=\mathit{p}\log(1+\mathcal{F})+\mathit{q}\log(1-\mathcal{F})
\end{align}
\end{proof}
The following main theorem then derives the optimum Kelly fraction $\mathcal{F}=\mathcal{F}_{K}$ as a critical point of the utility function $\bm{\mathsf{U}}(\mathcal{F},p)$.
\begin{thm}
The utility function $\bm{\mathsf{U}}(\mathcal{F},p)$ is maximised or optimised for
\begin{align}
\mathcal{F}=\mathcal{F}_{K}=|\mathit{p}-\mathit{q}|
\end{align}
\end{thm}
\begin{proof}
Taking the derivative with respect to $\mathcal{F}$ gives the maximum of the curve so that
\begin{align}
\frac{d \bm{\mathsf{U}}(\mathcal{F},p)}{d\mathcal{F}}\bigg|_{\mathcal{F}=\mathcal{F}_{K}}=0
\end{align}
which is
\begin{align}
\frac{d \bm{\mathsf{U}}(\mathcal{F},p)}{d\mathcal{F}}\bigg|_{\mathcal{F}=\mathcal{F}_{K}}
=\frac{\mathit{p}}{1+\mathcal{F}_{K}}-\frac{\mathit{q}}{1-\mathcal{F}_{K}}=\frac{\mathit{p}-\mathit{q}-\mathcal{F}_{K}}{(1+\mathcal{F}_{K})(1-\mathcal{F}_{K})}=0
\end{align}
so that $\mathcal{F}_{K}=\mathit{p}-\mathit{q}$. The 2nd derivative is
\begin{align}
\frac{d^{2}\bm{\mathsf{U}}(\mathcal{F},p)}{d\mathcal{F}^{2}}=-\frac{\mathit{p}}{(1+\mathcal{F})^{2}}-\frac{\mathit{q}}{(1-\mathcal{F})^{2}} <0
\end{align}
so the extremum is a maximum or peak of the function.
\end{proof}
\begin{lem}
Given the optimum Kelly strategy with $\mathcal{F}_{K}=|\mathit{p}-\mathit{q}|$ then the optimised differential growth factor $\bm{\mathsf{U}}(\mathcal{F}^{*})$ can be expressed in terms of the Shannon informational entropy $\mathscr{H}(\mathit{p},\mathit{q})$ such that
\begin{align}
\bm{\mathsf{U}}(\mathcal{F}_{K},p)=\log(2)-\mathscr{H}(\mathit{p},\mathit{q})=log 2+\mathit{p}\log\mathit{p}+\mathit{q}\log\mathit{q}
\end{align}
\end{lem}
\begin{proof}
From (3.17)
\begin{align}
\bm{\mathsf{U}}(\mathcal{F}_{K},p)&=\mathit{p}\log(1+{\mathcal{F}}^{*})+\mathit{q}\log(1-{\mathcal{F}}^{*})\nonumber\\&
=\mathit{p}\log(1+\mathit{p}-\mathit{q})+\mathit{q}\log(1-\mathit{p}+\mathit{q})\nonumber\\&
=\mathit{p}\log(1+\mathit{p}-1+\mathit{p})+\mathit{q}\log(1-1+\mathit{q}+\mathit{q})\nonumber\\&
=\mathit{p}\log(2\mathit{p})+\mathit{q}\log(2\mathit{q})\nonumber\\&
=\mathit{p}\log(2)+\mathit{p}\log(\mathit{p})+\mathit{q}\log(2)+\mathit{q}\log(\mathit{q})\nonumber\\&
=(\mathit{p}+\mathit{q})\log(2)+\mathit{p}\log(\mathit{p}+\mathit{q}\log(\mathit{q})=\log(2)+\mathit{p}\log(\mathit{p})+\mathit{q}\log(\mathit{q}) \nonumber\\& \equiv \log(2)-\mathscr{H}(\mathit{p},\mathit{q})
\end{align}
\end{proof}
It is easily shown that $\bm{\mathsf{U}}(\mathcal{F},p)$ is a concave function on an interval $[0,\mathcal{F}_{*}]\subset [0,1]$.
\begin{cor}
Since the curve peaks at $\mathcal{F}_{K}\in[0,1]$ then one must always have $p>1/2$. If $p<1/2$ then $\mathcal{F}_{K}<0$ so
$\mathcal{F}_{K}\notin [0,1]$ which is a contradiction. Also, the utility function $\bm{\mathsf{U}}(\mathcal{F}_{K},p)$ is zero at $p=1/2$ since $\bm{\mathsf{U}}(\mathcal{F}_{K},\tfrac{1}{2})=log(2)-\mathscr{H}(\tfrac{1}{2},\tfrac{1}{2})=\log(2)-\log(2)=0$ and so there will be zero net growth.
\end{cor}
\begin{cor}
The utility function  vanishes at $\mathcal{F}=0$ and $\mathcal{F}=\mathcal{F}_{*}$. First $\mathbb{G}(0)=plog(1)+q\log(1)=0$. Then
\begin{align}
\bm{\mathsf{U}}(\mathcal{F}_{*},p)=\mathit{p}\log(1+{\mathcal{F}}_{*})+\mathit{q}\log(1-\mathcal{F}_{*})=0
\end{align}
so that
\begin{align}
\frac{\log(1+\mathcal{F}_{*})}{\log(1-\mathcal{F}_{*})}=-\frac{\mathit{q}}{\mathit{p}}=\frac{(\mathit{p}-1)}{\mathit{p}}
\end{align}
This can be solved numerically for $\mathit{p}>\frac{1}{2}$.
\end{cor}
\begin{cor}
The utility function tends to negative infinity as $\mathcal{F} \rightarrow 1$ so that $\bm{\mathsf{U}}(1,p)=-\infty$.
\end{cor}
\begin{prop}
The GFC $\bm{\mathsf{U}}(\mathcal{F},p)$ is defined for the interval $[0,1]$. so $\mathcal{F}\in[0,1]$ and $\mathcal{F}_{K}=(p-q)$ for $\mathit{p}>q$ with $\mathit{p}>\tfrac{1}{2}$.
The unit interval can be partitioned as
\begin{align}
[0,1]=[0,\mathcal{F}_{K}]\bigcup(\mathcal{F}_{K},\mathcal{F}_{*}]\bigcup[\mathcal{F}_{*}]\bigcup(\mathcal{F}_{*},1]= \mathbb{I}\bigcup\mathbb{J}\bigcup\mathbb{K}
\end{align}
Then:
\begin{enumerate}
\item The utility function $\bm{\mathsf{U}}(\mathcal{F},p)>0$ for $\mathcal{F}\in[0,\mathcal{F}_{K}]\bigcup(\mathcal{F}_{K},\mathcal{F}_{*}]=\mathbb{I}\bigcup \mathbb{J}$.
\item $\bm{\mathsf{U}}(\mathcal{F},p)<0 $ for $\mathcal{F}\in(\mathcal{F}_{*},1]$.
\item $\bm{\mathsf{U}}(\mathcal{F}_{*},p)=0$ and $\bm{\mathsf{U}}(0,p)=0$
\item $\bm{\mathsf{U}}(\mathcal{F},p)$ is maximised or peaked at $\mathcal{F}=\mathcal{F}_{K}=(\mathit{p}-\mathit{q})$ so that $[d\bm{\mathsf{U}}(\mathcal{F})/d\mathcal{F}]_{\mathcal{F}=\mathcal{F}_{K}}=0$.
\item $\bm{\mathsf{U}}(1,p)=-\infty$
\end{enumerate}
Hence
\begin{align}
[0,1]=\underbrace{[0,\mathcal{F}_{K}]\bigcup(\mathcal{F}_{K},\mathcal{F}_{*})}_{\bm{\mathsf{U}}(\mathcal{F},p)\ge 0}\bigcup\underbrace{[\mathcal{F}_{*}]}_{\mathsf{U}(\mathcal{F}_{*},p)=0}
\bigcup\underbrace{(\mathcal{F}_{*},1]}_{\bm{\mathsf{U}}(\mathcal{F},p)<0}= \mathbb{I}\bigcup\mathbb{J}\bigcup\mathbb{K}
\end{align}
\end{prop}
The typical form of the utility function $\bm{\mathsf{U}}(\mathcal{F},p)$ is now plotted in Fig.1 for $\mathit{p}=0.60,\mathit{q}=0.40$ and it can be seen that\newline
\begin{figure}
\begin{center}
\hspace{-1cm}
\includegraphics[height=3.0in,width=5.5in]{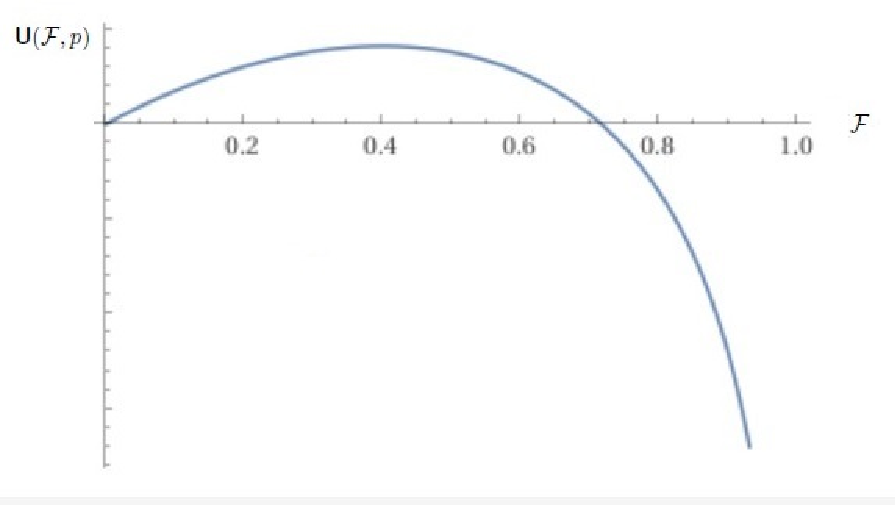}
\caption{Utility function $\bm{\mathsf{U}}(\mathcal{F},p)$ versus $\mathcal{F}$ with a peak at Kelly fraction $\mathcal{F}_{K}=0.4$}
\end{center}
\end{figure}
These values are chosen to illustrate the full form of the plot from $\mathcal{F}=0$ to $\mathcal{F}=1$.
\begin{rem}
\begin{enumerate}
The following observations can be made \textbf{[2]}:
\item If $\bm{\mathsf{U}}(\mathcal{F},p)>0$ then $\lim_{N\rightarrow\infty}\mathscr{W}(N)=\infty$  a.s, and for all $M>0$ $\bm{\mathsf{P}}(\inf \mathscr{W}(N)>M]=1$. The growth coefficient is positive and so the wealth grows on average. The fastest rate of growth which maximises $\bm{\mathsf{U}}(F)$ occurs when $\mathcal{F}=\mathcal{F}_{K}=\mathit{p}-\mathit{q}$.
\item If $\bm{\mathsf{U}}(\mathcal{F},p)<0$ then the utility function is negative and wealth is gradually depleted with increasing N, even with a favourable edge $\mathit{p}>\mathit{q}$. Then
$\lim_{N\rightarrow\infty}\mathscr{W}(N)=0$ a.d and for all very small $\epsilon>0$ one has $\bm{\mathsf{P}}(\lim_{N\rightarrow\infty}\mathscr{W}(N)\le \epsilon]=1$
\item If $\mathcal{F}=\mathcal{F}_{*}$ then $\bm{\mathsf{U}}(\mathcal{F}_{*},p)=0$ so that $\mathscr{W}(N)$ oscillates randomly between zero and infinity such that
\begin{align}
&\lim_{N\rightarrow\infty}\sup \mathscr{W}(N)=\infty, ~a.s\nonumber\\&
\lim_{N\rightarrow\infty}\sup \mathscr{W}(N)=0~a.s\nonumber
\end{align}
\item If $\mathcal{F}=1$ then $\bm{\mathsf{U}}(1)=-\infty$ and the ruin probability is absolute for any $p\in[\frac{1}{2},1)$. Any initial $\mathscr{W}(0)$ will be wiped out with increasing N.
In consecutive bets are made, and at each N, the amount $\mathscr{W}(N-1)$ is wagered then the probability of N wins with no losses is
\begin{align}
\bm{\mathsf{P}}(no~losses~in~N~bets)=\mathit{p}^{N}
\end{align}
The probability that the entire betted bankroll/money is lost at the $N^{th}$ bet is then
\begin{align}
\bm{\mathsf{P}}(\mathscr{W}(N)=0)=1-\mathit{p}^{N}
\end{align}
Hence, the ruin probability is absolute or unity for large N even for a highly favourable game with $p\gg\frac{1}{2}$ since $p<1$ and $\lim_{N\rightarrow 0} x^{N}=0$ for $0<x<1$ so
\begin{align}
\lim_{N\rightarrow\infty}\bm{\mathsf{P}}[\mathscr{W}(N)=0]=\lim_{N\rightarrow\infty}[1-\mathit{p}^{N}]=1
\end{align}
\item However if $\mathcal{F}=1$ and $p=1$ (e.g, a double-headed coin) then there is zero probability of ruin and the game is fully deterministic with no randomness or uncertainty.
Each outcome is then $\mathscr{Z}=1$.
\begin{align}
\lim_{N\rightarrow\infty}\bm{\mathsf{P}}[\mathscr{W}(N)=0]=\lim_{N\rightarrow\infty}[1-p^{N}]
=\lim_{N\rightarrow\infty}[1- 1^{N}]=0
\end{align}
The wealth $\mathscr{W}(N)$ then doubles with each bet and grows exponentially since a win is now 100 percent certain at each bet. The evolution of the liquidity or wealth is then
\begin{align}
&\mathscr{W}(N)=\mathscr{W}(0)\prod_{I=1}^{N}(1+\mathcal{F}\mathscr{Z}(I))=\mathscr{W}(0)(1+\mathcal{F})^{N}=\mathscr{W}(0) 2^{N}\nonumber\\&\equiv \mathscr{W}(0)\exp\log 2^{N}=\mathscr{W}(0)\exp(N\log(2))\sim \mathscr{W}(0)\exp(0.693 N)
\end{align}
so that the liquidity then grows exponentially.
\end{enumerate}
\end{rem}
The larger is $p>\frac{1}{2}$ then the larger is the growth factor or utility  $\bm{\mathsf{U}}(\mathcal{F},p)$ for all $\mathcal{F}\in [0,1]$.
\begin{lem}
Let $\bm{\mathsf{U}}(\mathcal{F},\mathit{p})$ be the utility function for $p>\frac{1}{2}$ and let $\bm{\mathsf{U}}(\mathcal{F},\bar{p})$ the function
for $\bar{p}<{p}$, with $\mathcal{F}\in[0,1]$. Then
\begin{align}
\bm{\mathsf{U}}(\mathcal{F},p) > \bm{\mathsf{U}}(\mathcal{F},\widehat{p}),~~{p}> \widehat{p}>\frac{1}{2}
\end{align}
\end{lem}
\begin{proof}
Define the difference or separation as
\begin{align}
\Delta\bm{\mathsf{U}}(\mathcal{F},p)=\bm{\mathsf{U}}(\mathcal{F},{p})-\bm{\mathsf{U}}({\mathcal{F}},\widehat{p})
\end{align}
where
\begin{align}
&\bm{\mathsf{U}}(\mathcal{F},p)=\mathit{p}\log(1+\mathcal{F})+(1-\mathit{p})\log(1-\mathcal{F})\\&
\bm{\mathsf{U}}(\mathcal{F},\widehat{p})=\widehat{p}\log(1+\mathcal{F})+(1-\widehat{p})\log(1-\mathcal{F})
\end{align}
Then
\begin{align}
&\bm{\mathsf{U}}(\mathcal{F},p)=(\mathit{p}-\widehat{p})\log(1+\mathcal{F})+(\widehat{p}-{p})\log(1-\mathcal{F})
\end{align}
Now analyse the signs of the terms. Since $\mathit{p}>\widehat{p}$ then $(\mathit{p}-\widehat{p})>0$. For all $\mathcal{F}\in[0,1], \log(1+\mathcal{F})\ge 0$.
However $(\widehat{p}-\mathit{p})<0$ for $\mathit{p}>\widehat{p}$ and $\log(1-\mathcal{F})\le 0$ for $\mathcal{F}\in[0,1]$. Hence
\begin{align}
&\Delta\bm{\mathsf{U}}(\mathcal{F})=\underbrace{(\mathit{p}-\widehat{p})}_{+VE}\underbrace{\log(1+\mathcal{F})}_{+VE}+\underbrace{(\widehat{p}-\mathit{p})}_{-VE}\underbrace{\log(1-\mathcal{F})}_{-VE}\nonumber\\&
=\underbrace{(\mathit{p}-\widehat{p})\log(1+\mathcal{F})}_{+VE}+\underbrace{(\widehat{p}-\mathit{p})\log(1-\mathcal{F})}_{+VE}>0
\end{align}
Therefore $\Delta\bm{\mathsf{U}}(\mathcal{F})>0$ and so $\bm{\mathsf{U}}(\mathcal{F};\mathit{p})>\bm{\mathsf{U}}(\mathcal{F};\widehat{p})$. Note also that $\bm{\mathsf{U}}(1;\widehat{p})=
\bm{\mathsf{U}}(1;{p})=-\infty$.
\end{proof}
This is to be expected since a larger value of $\mathit{p}>\frac{1}{2}$ leads to more favourable outcome on average and so the growth factor or utility function
will be larger for all values of $\mathcal{F}$. The advantage of a higher $\mathit{p}$ also grows with increasing $\mathcal{F}$. This is illustrated in Figure 2. Hence, the peaks are larger for larger $\mathit{p}$ and the Kelly fraction is larger. Here we can see that $\bm{\mathsf{U}}(\mathcal{F},\mathit{p})>\bm{\mathsf{U}}(\mathcal{F},\widehat{p}^{*})>\bm{\mathsf{U}}(\mathcal{F},\widehat{\widehat{p}})$ where $\mathit{p}=0.52,\widehat{p}=0.515,\widehat{\widehat{p}}=0.51.$ and all curves will converge to $-\infty$ as $\mathcal{F}\rightarrow 1$.
\begin{figure}
\begin{center}
\includegraphics[height=2.5in,width=6.5in]{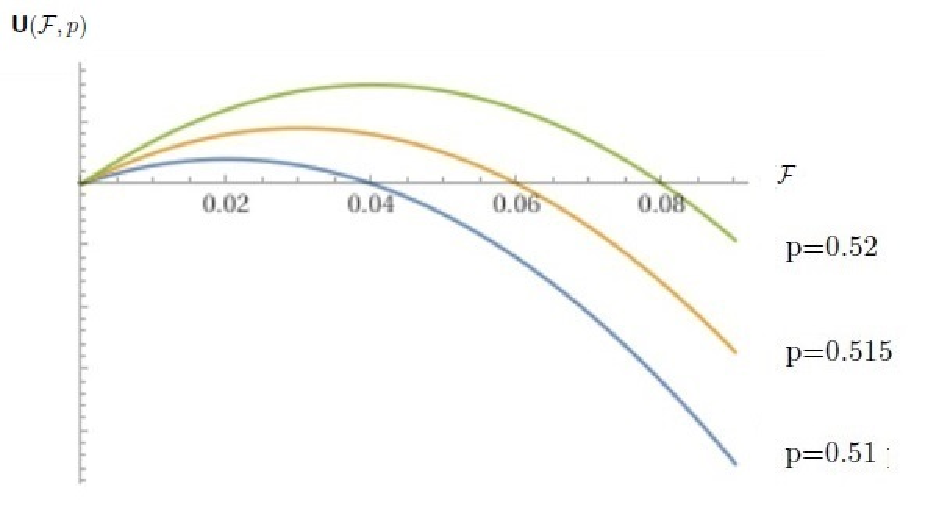}
\caption{Utility function curves $\bm{\mathsf{U}}(\mathcal{F},p)$ for p=0.52, p=0.515, p=0.51}
\end{center}
\end{figure}
Note that the graphs are not exactly symmetric about $\mathcal{F}_{K}$ but nonetheless are still very close to it. It may seem that $\mathcal{F}_{*}=2\mathcal{F}$ but this not the actually the case. However, for small $\mathcal{F}$, and $\mathcal{F}$ is usually small since $\mathit{p}$ is usually not much greater than $0.5$, one has the robust approximation $\mathcal{F}_{*}\sim 2\mathcal{F}$.
\begin{lem}
For small $\mathcal{F}$, one has the approximations
\begin{align}
\mathcal{F}\sim 2 \mathcal{F}=2(\mathit{p}-\mathit{q})=2(2\mathit{p}-1)
\end{align}
and
\begin{align}
\mathcal{F}_{*}\sim 2 \mathcal{F}+\epsilon=2 \mathcal{F}_{K}+ \frac{\mathcal{F}_{K}^{2}}{\frac{3}{8}-\mathcal{F}_{K}^{2}}
\end{align}
\begin{proof}
\begin{align}
\bm{\mathsf{U}}(\mathcal{F}_{*},p)=\mathit{p}\log(1+\mathcal{F}_{c})+\mathit{q}\log(1-\mathcal{F}_{*})=0
\end{align}
Expanding the logs via Taylor series
\begin{align}
&\log(1+\mathcal{F}_{*})\sim \mathcal{F}_{*}-\frac{1}{2}\mathcal{F}_{*}^{2}+\frac{1}{3}\mathcal{F}_{*}^{3}\\&
\log(1-\mathcal{F}_{*})\sim -\mathcal{F}_{*}-\frac{1}{2}\mathcal{F}_{*}^{2}-\frac{1}{3}\mathcal{F}_{*}^{3}
\end{align}
Then
\begin{align}
&\bm{\mathsf{U}}(\mathcal{F}_{*},p)=\mathit{p}(\mathcal{F}_{*}-\frac{1}{2}\mathcal{F}_{*}^{2}+\frac{1}{3}\mathcal{F}_{*}^{3})+\mathit{q}(-_{C}-\frac{1}{2}\mathcal{F}_{*}^{2}-\frac{1}{3}\mathcal{F}_{*}^{3})\nonumber\\&
=\mathit{p}\mathcal{F}_{*}-\frac{1}{2}\mathit{p}\mathcal{F}_{*}^{2}+\frac{1}{3}\mathit{p}\mathcal{F}_{*}-\mathit{q}\mathcal{F}_{*}-\frac{1}{2}\mathit{q}\mathcal{F}_{*}^{2}-\frac{1}{3}\mathit{q}\mathcal{F}_{*}^{3}\nonumber\\&
=\mathcal{F}_{*}(\mathit{p}-\mathit{q})-\frac{1}{2}\mathcal{F}_{*}^{2}(\mathit{p}+\mathit{q})+\frac{1}{3}\mathcal{F}_{*}^{3}(\mathit{p}-\mathit{q})\nonumber\\&
=\mathcal{F}_{*}(2\mathit{p}-1)-\frac{1}{2}\mathcal{F}_{*}^{2}+\frac{1}{3}\mathcal{F}_{*}^{3}(2\mathit{p}-1)\nonumber\\&
\end{align}
Now dropping the cubic term, which is very small, we have the quadratic approximation
\begin{align}
\mathcal{F}_{*}(2\mathit{p}-1)-\frac{1}{2}\mathcal{F}_{*}^{2}\sim 0
\end{align}
or
\begin{align}
(2\mathit{p}-1)-\frac{1}{2}\mathcal{F}_{*}\sim 0
\end{align}
which gives $\mathcal{F}_{*}\sim 2(2\mathit{p}-1)=2\mathcal{F}_{K}$. For the error term $\epsilon$, if $\mathcal{F}_{*}=2\mathcal{F}_{K}+\epsilon$ then
\begin{align}
(2\mathit{p}-1)-\frac{1}{2}\mathcal{F}_{*}+\frac{1}{3}\mathcal{F}_{*}^{2}(2\mathit{p}-1)\sim 0
\end{align}
becomes
\begin{align}
(2\mathit{p}-1)-\frac{1}{2}(2\mathcal{F}_{K}+\epsilon)+\frac{1}{3}(2\mathcal{F}_{K}+\epsilon)^{2}(2\mathit{p}-1)\sim 0
\end{align}
Now $\mathcal{F}_{K}=(2\mathit{p}-1)$ so that
\begin{align}
\mathcal{F}_{K}-\frac{1}{2}(2\mathcal{F}_{K}+\epsilon)+\frac{1}{3}(2\mathcal{F}_{K}+\epsilon)^{2}\mathcal{F}_{K}=-\frac{1}{2}\epsilon+\frac{4}{3}\mathcal{F}_{K}^{3}+
\frac{4}{3}\mathcal{F}_{K}^{2}\epsilon+\frac{1}{3}\mathcal{F}_{K}\epsilon^{2}
\end{align}
Keeping terms first order in $\epsilon$
\begin{align}
-\frac{1}{2}\epsilon+\frac{4}{3}\mathcal{F}_{K}^{3}+\frac{4}{3}\mathcal{F}_{K}^{2}\epsilon \sim 0
\end{align}
or
\begin{align}
\epsilon(-\frac{1}{2}+\frac{4}{3}\mathcal{F}_{K}^{2})=-\frac{4}{3}\mathcal{F}_{K}^{3}
\end{align}
Now solving for $\epsilon$
\begin{align}
\epsilon\sim \frac{\frac{4}{3}\mathcal{F}_{K}^{3}}{\frac{1}{2}-\frac{4}{3}\mathcal{F}_{K}^{2}}=\frac{\mathcal{F}_{K}^{3}}{\frac{3}{8}-\mathcal{F}_{K}^{2}}
\end{align}
\end{proof}
For typical values of $\mathcal{F}_{K}$, the term $\epsilon$ is very small. For example, for $\mathcal{F}_{K}=0.04 $, then $\epsilon$ is approximately
$0.0001714$ and can be ignored.
\end{lem}
\section{Submartingale and supermartingale regimes}
It is now formally proved that the process $\lbrace \mathscr{W}(I)\rbrace_{I=1}^{N}$ is submartingale when $\bm{\mathsf{U}}(\mathcal{F},p)>0$ and a supermartingale
when $\bm{\mathsf{U}}(p,\mathcal{F})<0$. Hence, on average, $\mathscr{W}(N)$ either grows or decays as N increases. For the submartingale case, one can apply Doob's submartinagle inequality and decomposition theorem \textbf{[20]}. A generic stochastic process $\lbrace \mathscr{S}(I)\rbrace_{I=1}^{N}$ or $\mathscr{S}(N)$ is a submartingale with respect to a filtration $\mathcal{F}_{N }$ if
$\bm{\mathsf{E}}[\mathscr{S}(N+1)|\mathfrak{F}_{N}]\ge \mathscr{S}(N)$, a martingale if $\bm{\mathsf{E}}[\mathscr{S}(N+1)|\mathfrak{F}_{N}]=\mathscr{S}(N)$ and a supermartingale if
$\bm{\mathsf{E}}[\mathscr{S}(N+1)|\mathfrak{F}_{N}]\le \mathscr{S}(N)$
\begin{thm}(\textbf{Submartingale and supermartingale regimes})\newline
The liquidity/wealth accumulated at the $N^{th}$ trial  is
\begin{align}
\mathscr{W}(N)=\mathscr{W}(0)\prod_{I=1}^{N}(1+\mathcal{F}\mathscr{Z}(N))=\mathscr{W}(N-1)(1+\mathcal{F}\mathscr{Z}(N))
\end{align}
with $\bm{\mathsf{P}}(\mathscr{N}=1)=p$ and $\bm{\mathsf{P}}(\mathscr{Z}(N)=-1)=q=1-p$ for all N. For all
\begin{align}
\mathcal{F}\in [0,1]=(0,\mathcal{F}_{K}]\bigcup(\mathcal{F}_{K},\mathcal{F}_{*})\bigcup[\mathcal{F}_{*}]\bigcup(\mathcal{F}_{*},1]
\end{align}
and for $p>1/2$ the utility function is
\begin{align}
\bm{\mathsf{U}}(\mathcal{F},p)=\bm{\mathsf{E}}\left[\frac{1}{N}\log \frac{\mathscr{W}(N)}{\mathscr{W}(0)}\right]=p\log(1+\mathcal{F})+(1-p)\log(1-\mathcal{F})
\end{align}
where $\bm{\mathsf{U}}(\mathcal{F}_{*},p)=0$ and with a critical point (maximum) at the Kelly fraction $\mathcal{F}=\mathcal{F}_{K}=2p-1$. Then:
\begin{enumerate}
\item If $\mathcal{F}\in(0,\mathcal{F}_{*})$ then $\bm{\mathsf{U}}(\mathcal{F},p)>0$ for all $p>1/2$ and $p\le 1$ and $\mathscr{W}(N)$ is a submartingale on $(0,\mathcal{F}_{*})$ and
\begin{align}
\bm{\mathsf{E}}[\mathscr{W}(N+1)|\mathfrak{F}_{N}]\ge \mathscr{W}(N)
\end{align}
\item If $\mathcal{F}\in (\mathcal{F}_{*},1)$ then $\bm{\mathsf{U}}(\mathcal{F},p)<0$ for all $p>1/2$ and $p\le 1$ and $\mathscr{W}(N)$ is a supermartingale on $(\mathcal{F}_{*},1]$ and
\begin{align}
\bm{\mathsf{E}}[\mathscr{W}(N+1)|\mathfrak{F}_{N}]\le \mathscr{W}(N)
\end{align}
\item If $\mathcal{F}\in (\mathcal{F}_{*},1)$ then $\bm{\mathsf{U}}(\mathcal{F},p)=0$ for all $p>1/2$ and $p\le 1$ and $\mathscr{W}(N)$ is a martingale and
\begin{align}
\bm{\mathsf{E}}[\mathscr{W}(N+1)|\mathfrak{F}_{N}]=\mathscr{W}(N)
\end{align}
\end{enumerate}
\end{thm}
\begin{proof}
\textbf{To prove (1)}:~Begin with $\mathscr{W}(N+1)=\mathscr{W}(N)(1+\mathcal{F}\mathscr{Z}(N+1))$ where $\mathscr{W}(N)$ is an adapted process, adapted to the filtration $\mathfrak{F}_{N}$ so
\begin{align}
(\mathscr{W}(N+1)|\mathfrak{F}_{N})=\mathscr{W}(N)(1+\mathcal{F}\mathscr{Z}(N+1)|\mathfrak{F}_{N})
\end{align}
Taking the log of both sides
\begin{align}
\log(\mathscr{W}(N+1)|\mathfrak{F}_{N})&=\log(\mathscr{W}(N)(1+\mathcal{F}\mathscr{Z}(N+1)|\mathfrak{F}_{N}))\nonumber\\&=\log(\mathscr{W}(N))+\log(1+\mathcal{F}\mathscr{Z}(N+1)|\mathfrak{F}_{N}))
\end{align}
Now taking the expectation $\bm{\mathsf{E}}[\bullet]$
\begin{align}
&\log(\mathscr{W}(N+1)|\mathfrak{F}_{N})=\bm{\mathsf{E}}[\log(\mathscr{W}(N))]+\bm{\mathsf{E}}[\log(1+\mathcal{F}\mathscr{Z}(N+1)|\mathfrak{F}_{N}))]\nonumber\\&
=\log(\mathscr{W}(N))+\bm{\mathsf{E}}[\log(1+\mathcal{F}\mathscr{Z}(N+1)|\mathfrak{F}_{N}))]\nonumber\\&
=\log(\mathscr{W}(N))+p\log(1+\mathcal{F})+(1-p)\log(1-\mathcal{F})=\log(\mathscr{W}(N))+\bm{\mathsf{U}}(\mathcal{F},p)
\end{align}
and where $\mathscr{W}(N)$ is a known quantity at the $(N+1)^{th}$ trial. Now apply Jensen's inequality to the lhs. If $f$ is concave then $\bm{\mathsf{E}}[f(\mathscr{R})\le f(\bm{\mathsf{E}}[\mathscr{R}])$ for a random variable $\mathscr{R}$ and a concave function $f$. Then
\begin{align}
\log(\mathscr{W}(N+1)|\mathfrak{F}_{N})\le \log\bm{\mathsf{E}}[\mathscr{W}(N+1)|\mathfrak{F}_{N}]
\end{align}
which is
\begin{align}
\log(\mathscr{W}(N))+\bm{\mathsf{U}}(\mathcal{F},p) \le \log\bm{\mathsf{E}}[\mathscr{W}(N+1)|\mathfrak{F}_{N}]
\end{align}
Now taking the exponential of both sides
\begin{align}
\mathscr{W}(N))\exp(\bm{\mathsf{U}}(\mathcal{F},p))\le \bm{\mathsf{E}}[\mathscr{W}(N+1)|\mathfrak{F}_{N}]
\end{align}
or
\begin{align}
\bm{\mathsf{E}}[\mathscr{W}(N+1)|\mathfrak{F}_{N}]\ge \mathscr{W}(N)\exp(\bm{\mathsf{U}}(\mathcal{F},p))
\end{align}
Suppose now the following inequality holds
\begin{align}
\bm{\mathsf{E}}[\mathscr{W}(N+1)|\mathfrak{F}_{N}]\ge \mathscr{W}(N)\exp(\bm{\mathsf{U}}(\mathcal{F},p))\ge \mathscr{W}(N)
\end{align}
If this is true then $\mathscr{W}(N)$ is a submartingale. This requires
\begin{align}
\mathscr{W}(N)\exp(\bm{\mathsf{U}}(\mathcal{F},p))\ge \mathscr{W}(N)
\end{align}
or $\exp(\bm{\mathsf{U}}(\mathcal{F},p))\ge 1$. From the basic properties of the exponential function, $\exp(\bm{\mathsf{U}}(\mathcal{F},p))\ge 1$ iff $\bm{\mathsf{U}}(\mathcal{F},p)>0$ which is the case if
$\mathcal{F}\in(0,\mathcal{F}_{*})$. Hence $\bm{\mathsf{E}}[\mathscr{W}(N+1)|\mathfrak{F}_{N}]\ge \mathscr{W}(N)$  and so the process is a submartingale and increases on average.\newline
\textbf{To prove (2)}: Begin with (4.12)
\begin{align}
\mathscr{W}(N))\exp(\bm{\mathsf{U}}(\mathcal{F},p))\le \bm{\mathsf{E}}[\mathscr{W}(N+1)|\mathfrak{F}_{N}]
\end{align}
Suppose the following inequality is true
\begin{align}
\mathscr{W}(N))\exp(\bm{\mathsf{U}}(\mathcal{F},p))\le \bm{\mathsf{E}}[\mathscr{W}(N+1)|\mathfrak{F}_{N}]\le \mathscr{W}(N)
\end{align}
then the process is a supermartingale. This inequality is true if $\exp(\bm{\mathsf{U}}(\mathcal{F},p))\le 1$ which requires $\bm{\mathsf{U}}(\mathcal{F},p)<0$ and this is the case for $\mathcal{F}\in(\mathcal{F}_{*},1]$. Hence $\bm{\mathsf{E}}[\mathscr{W}(N+1)|\mathfrak{F}_{N}]\le \mathscr{W}(N)$ and the process is a supermartinagle and decreases on average.\newline
\textbf{To prove (3)}: Beginning again with
\begin{align}
\mathscr{W}(N))\exp(\bm{\mathsf{U}}(\mathcal{F},p))\le \bm{\mathsf{E}}[\mathscr{W}(N+1)|\mathfrak{F}_{N}]\le \mathscr{W}(N)
\end{align}
If $\mathcal{F}=\mathcal{F}_{*}$ then $\bm{\mathsf{U}}(\mathcal{F}_{*},p)=0$ and $\exp(\bm{\mathsf{U}}(\mathcal{F}_{*},p))=0$ giving
\begin{align}
\mathscr{W}(N))\le \bm{\mathsf{E}}[\mathscr{W}(N+1)|\mathfrak{F}_{N}]\le \mathscr{W}(N)
\end{align}
This can only be satisfied in the equality
\begin{align}
\mathscr{W}(N))=\bm{\mathsf{E}}[\mathscr{W}(N+1)|\mathfrak{F}_{N}]=\mathscr{W}(N)
\end{align}
and so the process in this case is a martingale.
\end{proof}
The following estimate is now given for the expectation $\bm{\mathsf{E}}[\mathscr{W}(N)]$.
\begin{lem}
\begin{align}
\bm{\mathsf{E}}[\mathscr{W}(N)]=\mathscr{W}(0)(1+\mathcal{F}(2p-1))^{N}
\end{align}
\end{lem}
\begin{proof}
\begin{align}
\mathscr{W}(N)=\mathscr{W}(0)\prod_{I=1}^{N}(1+\mathcal{F}\mathscr{Z}(N))
\end{align}
Taking the expectation
\begin{align}
&\bm{\mathsf{E}}[\mathscr{W}(N)]=\mathscr{W}(0)\bm{\mathsf{E}}\left[\prod_{I=1}^{N}(1+\mathcal{F}\mathscr{Z}(N))\right]\nonumber\\&
=\mathscr{W}(0)\prod_{I=1}^{N}[p(1+\mathcal{F})+q(1-\mathcal{F})]_{I}\nonumber\\&
=\mathscr{W}(0)\prod_{I=1}^{N}[p(1+\mathcal{F})+(1-p)(1-\mathcal{F})]_{I}\nonumber\\&
=\mathscr{W}(0)\prod_{I=1}^{N}[(1+2\mathcal{F}-\mathcal{F})]_{I}\nonumber\\&
=\mathscr{W}(0)\prod_{I=1}^{N}(1+\mathcal{F}(2p-1))_{I}\nonumber
=\mathscr{W}(0)(1+\mathcal{F}(2p-1))^{N}
\end{align}
since $(1+\mathcal{F}(2p-1))_{I}=(1+\mathcal{F}(2p-1))$ for all $I=1...N$.
\end{proof}
\begin{lem}
\begin{align}
\overline{\bm{\mathsf{E}}[\mathscr{W}(N)]}=\mathscr{W}(0)(1+p\mathcal{F})^{N}(1-q\mathcal{F})^{N}
\end{align}
\end{lem}
\begin{proof}
Since $\mathscr{W}(N)=\mathscr{W}(0)(1+\mathcal{F})^{U}(1-\mathcal{F})^{V}$ with $N=U+V=$ wins+losses, then the expectation is
\begin{align}
\overline{\bm{\mathsf{E}}[\mathscr{W}(N)]}=\mathscr{W}(0)\bm{\mathsf{E}}[(1+\mathcal{F})^{U}(1-\mathcal{F})^{V}]=\mathscr{W}(0)\bm{\mathsf{E}}[(1+\mathcal{F})^{U}]\bm{\mathsf{E}}[(1-\mathcal{F})^{V}]
\end{align}
Now using the binomial moment-generating functions (Appendix A)
\begin{align}
&\bm{\mathsf{M}}_{N}(U,\xi)=\bm{\mathsf{E}}[\exp(\xi U)]=(1-p+p\exp(\xi))^{N}\\&
\bm{\mathsf{M}}_{N}(V,\xi)=\bm{\mathsf{E}}[\exp(\xi V)]=(1-q+q\exp(\xi))^{N}
\end{align}
and setting $\xi=\log(1\pm\mathcal{F})$ gives
\begin{align}
&\bm{\mathsf{E}}[(1+\mathcal{F})^{U}]=(1+p\mathcal{F})^{N}\\&
\bm{\mathsf{E}}[(1-\mathcal{F})^{V}]=(1-q\mathcal{F})^{N}
\end{align}
and so (4.22) immediately follows.
\end{proof}
\begin{cor}
\begin{align}
\overline{\bm{\mathsf{E}}[\mathscr{W}(N)]}\sim \bm{\mathsf{E}}[\mathscr{W}(N)]
\end{align}
\end{cor}
\begin{proof}
\begin{align}
\overline{\bm{\mathsf{E}}[\mathscr{W}(N)]}=\mathscr{W}(0)(1+p\mathcal{F})^{N}(1-q\mathcal{F})^{N}=\mathscr{W}(0)[1+p\mathcal{F}-q\mathcal{F}-pq\mathcal{F}^{2}]^{N}
\end{align}
For small $\mathcal{F}$, which is true for $p>1/2$ and p close to $1/2$ then the term $pq\mathcal{F}^{2}$ is very small and can be ignored. For example, if $p=0.51$ then
$q=0.49$ and $\mathcal{F}=p-q=0.02$ giving $pq\mathcal{F}^{2}=0.00009996$. Then
\begin{align}
\overline{\bm{\mathsf{E}}[\mathscr{W}(N)]}&=\mathscr{W}(0)(1+p\mathcal{F})^{N}(1-q\mathcal{F})^{N}=\mathscr{W}(0)[1+p\mathcal{F}-q\mathcal{F}-pq\mathcal{F}^{2}]^{N}\nonumber\\&
\sim \mathscr{W}(0)(1+p\mathcal{F}-q\mathcal{F})^{N}=\mathscr{W}(0)(1+\mathcal{F}(p-q))^{N}\nonumber\\&
=\mathscr{W}(0)(1+\mathcal{F}(2p-1))^{N}=\bm{\mathsf{E}}[\mathscr{W}(N)]
\end{align}
\end{proof}
\begin{cor}
If $\mathcal{F}=\mathcal{F}_{K}=p-q$ then
\begin{align}
\bm{\mathsf{E}}[\mathscr{W}(N)]=\mathscr{W}(0)(1+\mathit{p}^{2}-pq)^{N}(1-pq+\mathit{q}^{2})^{N}
\end{align}
If $\mathcal{F}=p-q=0$ for $p=q=\tfrac{1}{2}$ then $\bm{\mathsf{E}}[\mathscr{W}(N)]=\mathscr{W}(0)$ since the game is not played within the Kelly strategy unless $p>q$.
\end{cor}
\begin{figure}[htb]
\begin{center}
\includegraphics[height=3.0in,width=6.0in]{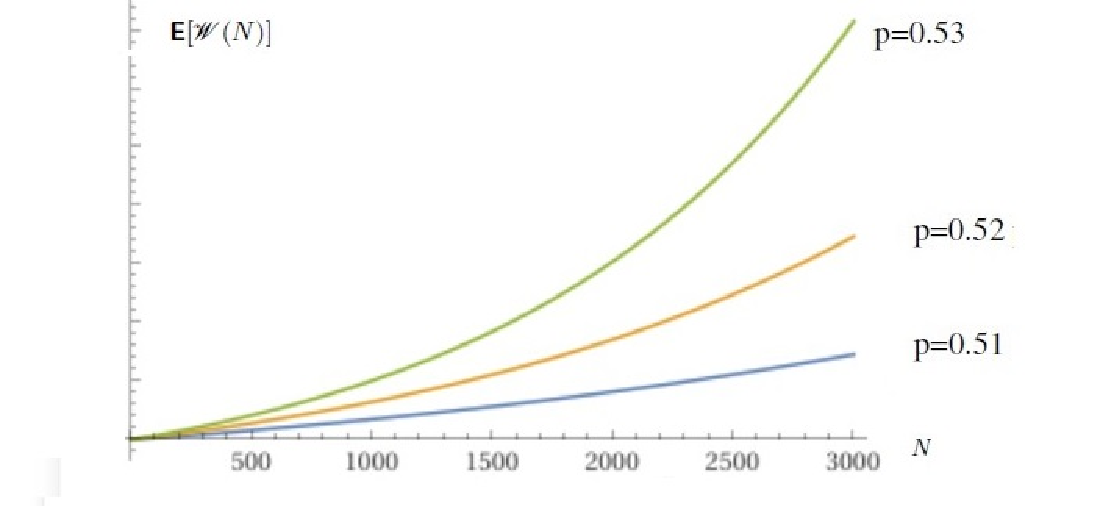}
\caption{Plots of $\bm{\mathsf{E}}[\mathscr{W}(N)]=\mathscr{W}(0)(1+\mathcal{F}_{K}(2p-1)-p(1-p)\mathcal{F}^{2})^{N}$ vs $N$ for $p=0.51, p=0.52, p=0.53$ for the Kelly fraction $\mathcal{F}_{K}=2p-1$}
\end{center}
\end{figure}
\subsection{Expectation in the supermartinagle realm}
The previous graph was for $\bm{\mathsf{E}}[\mathscr{W}(N)]$ and the Kelly fraction $\mathcal{F}=\mathcal{F}_{K}=p-q=2p-1$ where $\mathscr{W}(N)$ is a submartingale so the plots are monotonically
increasing with $N$. Suppose now the supermartingale realm is considered when $\mathcal{F}>\mathcal{F}_{*}$ or $\mathcal{F}\in(\mathcal{F}_{*},]$. For this choice of $\mathcal{F}$ one should find that the plots are monotonically decreasing with N, as is the case.
\begin{figure}[htb]
\begin{center}
\includegraphics[height=3.0in,width=6.0in]{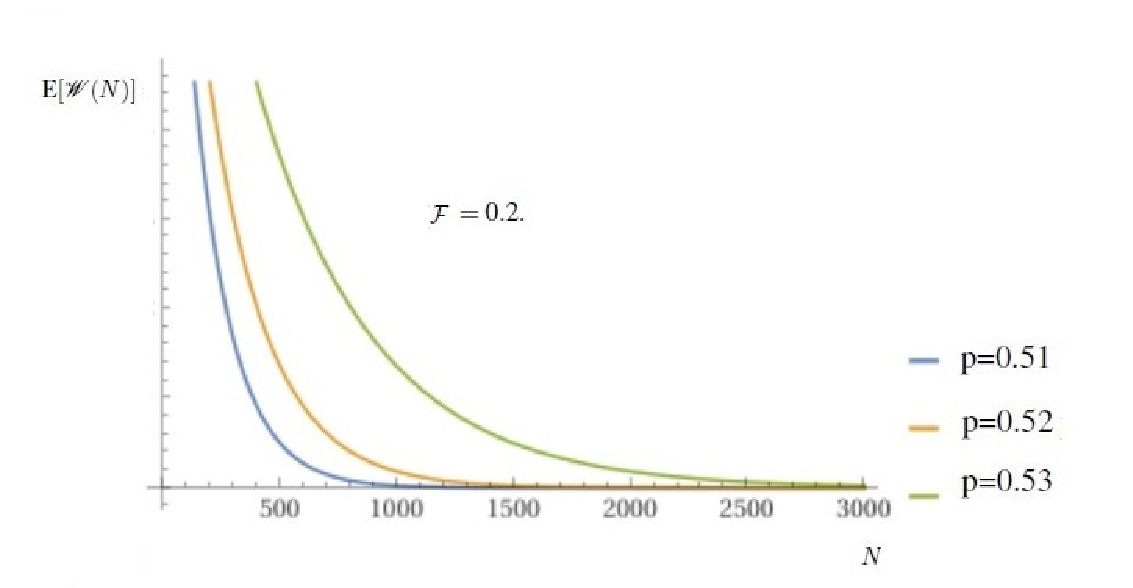}
\caption{Plots of $\bm{\mathsf{E}}[\mathscr{W}(N)]=\mathscr{W}(0)(1+\mathcal{F}(2p-1)-p(1-p)\mathcal{F}^{2}))^{N}$ vs $N$
for $p=0.51, p=0.52, p=0.53$ for $\mathcal{F}=0.2\gg\mathcal{F}_{*})$}
\end{center}
\end{figure}However, the full formula for $\bm{\mathsf{E}}[\mathscr{W}(N)]$ is used, that is $\bm{\mathsf{E}}[\mathscr{W}(N)]=\mathscr{W}(0)(1+\mathcal{F}(2p-1)-pq\mathcal{F}^{2})^{N}$, restoring the quadratic term. Choosing $\mathcal{F}\gg \mathcal{F}_{*}$ and $p=0.51,0.52,0.53$ as before with $\mathcal{F}=0.2$,the behavior of $\bm{\mathsf{E}}[\mathscr{W}(N)]$ is plotted against N. It can be seen that any initial $\mathscr{W}(0)$ is rapidly driven to zero as N increases since $\mathscr{W}(N)$ is now a supermartingale--any initial wealth or liquidity $\mathscr{W}(0)$ is guaranteed to be lost with unit probability when using a fraction $\mathcal{F}>\mathcal{F}_{*}$.
\newline
\subsection{Application of some classic submartinagle theorems}
For the regime where $\mathscr{W}(N)$ is a submartingale, the classic Doob submartingale theorems can also be applied \textbf{[20]}.
\begin{lem}
For large N, and $\mathcal{F}\in(0,\mathcal{F}_{*})$ with $\bm{\mathsf{U}}(\mathcal{F},p)$, the expectation of wealth grows exponentially so that
\begin{align}
\bm{\mathsf{E}}[\mathscr{W}(N)]\sim\mathscr{W}(0)\exp(N|p-q|\mathcal{F})
\end{align}
\end{lem}
\begin{proof}
\begin{align}
\bm{\mathsf{E}}\left[\frac{\mathscr{W}(N)}{\mathscr{W}(0)}\right]=(1+{p}\mathcal{F})^{N}(1-{q}\mathcal{F})^{N}
\end{align}
Then
\begin{align}
\log\bm{\mathsf{E}}\left[\frac{\mathscr{W}(N)}{\mathscr{W}(0)}\right]=N\log(1+\mathsf{p}\mathcal{F})+N\log(1-\mathsf{q}\mathcal{F})
\end{align}
Since $\mathsf{p}\mathcal{F}$ and $\mathsf{q}\mathcal{F}$ are very small the logs can be expanded as $\log(1\pm x)\sim \pm x$ so that
\begin{align}
&\log\bm{\mathsf{E}}\left[\frac{\mathscr{W}(N)}{\mathscr{W}(0)}\right]=N\log(1+\mathit{p}\mathcal{F})+N\log(1-\mathit{q}\mathcal{F})\nonumber\\&\sim (N\mathit{p}\mathcal{F}-N\mathit{q}\mathcal{F})=N(\mathit{p}\mathcal{F}-\mathit{q}\mathcal{F})
\end{align}
Then
\begin{align}
\bm{\mathsf{E}}\left[\frac{\mathscr{W}(N)}{\mathscr{W}(0)}\right] \sim \exp(N\mathcal{F}(\mathit{p}-\mathit{q}))
\end{align}
or
\begin{align}
\bm{\mathsf{E}}[\mathscr{W}(N)]\sim \mathscr{W}(0)\exp(N\mathcal{F}(\mathit{p}-\mathit{q}))
\end{align}
and so the expectation of the liquidity grows with N with $F\in (0,\mathcal{F}_{K}]$.
\end{proof}
When $\mathcal{F}\in(0,\mathcal{F}_{*})$ then $\mathscr{W}(N)$ is a submartingale and so classic submartingale theorems apply.
\begin{lem}(\textbf{Doob submartingale Theorem})~Let $\lambda>0$, then the probability that $\mathscr{W}(K)\ge\lambda$ for $K\le N$ is
\begin{align}
\bm{\mathsf{P}}\left(\sup_{K\le N}\mathscr{W}(K)\ge \lambda\right)\le \frac{\bm{\mathsf{E}}[\mathscr{W}(N)]}{\lambda}
\end{align}
Using the expressions already derived for $\bm{\mathsf{E}}[\mathscr{W}(N)]$ gives
\begin{align}
&\bm{\mathsf{P}}\left(\sup_{K\le N}\mathscr{W}(K)\ge \lambda\right)\le \frac{\mathscr{W}(0)(1+\mathcal{F}(2p-1))^{N}}{\lambda}\\&
\bm{\mathsf{P}}\left(\sup_{K\le N}\mathscr{W}(K)\ge \lambda\right)\le \frac{\mathscr{W}(0)[(1+p\mathcal{F})(1-q\mathcal{F})]^{N}}{\lambda}
\end{align}
However, the right-hand sides have to be less than or equal to one so that
\begin{align}
&\frac{\mathscr{W}(0)(1+\mathcal{F}(2p-1))^{N}}{\lambda}\le 1\\&
 \frac{\mathscr{W}(0)[(1+p\mathcal{F})(1-q\mathcal{F})]^{N}}{\lambda}\le 1
\end{align}
which imposes constraints on $\lambda$ such that
\begin{align}
&\lambda\ge \mathscr{W}(0)(1+\mathcal{F}(2p-1))^{N}\\&
\lambda\ge \mathscr{W}(0)[(1+p\mathcal{F})(1-q\mathcal{F})]^{N}
\end{align}
\end{lem}
\begin{cor}
At the Kelly fraction $\mathcal{F}=\mathcal{F}_{K}=2p-1=p-q$
\begin{align}
&\bm{\mathsf{P}}\left(\sup_{K\le N}\mathscr{W}(K)\ge \lambda\right)\le \frac{\mathscr{W}(0)(4p^{2}-4p+2)^{N}}{\lambda}=\frac{\mathscr{W}(f(p))^{N}}{\lambda}\\&
\bm{\mathsf{P}}\left(\sup_{K\le N}\mathscr{W}(K)\ge \lambda\right)\le \frac{\mathscr{W}(0)(4p^{4}-8p^{3}+9p^{2}-5p+2)^{N}}{\lambda}=\frac{\mathscr{W}(f(p))^{N}}{\lambda}
\end{align}
\end{cor}
\begin{cor}
The probability of reaching infinite wealth in a finite number of trials $N$ is zero and so there  can never be a 'blowup'. Hence
\begin{align}
\bm{\mathsf{P}}\left(\sup_{K\le N}\mathscr{W}(K)\right)=\infty)=0
\end{align}
\end{cor}
\begin{proof}
This follows immediately from (4.47)
\begin{align}
\bm{\mathsf{P}}\left(\sup_{K\le N}\mathscr{W}(K)=\infty\right)=\le \lim_{\lambda\rightarrow \infty}\frac{\bm{\mathsf{E}}[\mathscr{W}(N)]}{\lambda}=0
\end{align}
\end{proof}
The last theorem that can be applied is the submartingale decomposition theorem by Doob.
\begin{lem}(\textbf{Doob Submartingale Decomposition Theorem} Let $\lbrace\mathscr{S}(I)\rbrace_{I=1}^{N}$ be a submartingale, then there is a (Doob) decomposition into a martingale and a deterministic or predictable 'drift' term:
\begin{align}
\lbrace \mathscr{W}(I)\rbrace_{I=1}^{N}=\lbrace\mathscr{M}(I)\rbrace_{I=1}^{N}+\lbrace A(I))\rbrace_{I=1}^{N}
\end{align}
or
\begin{align}
\mathscr{S}(N)=\mathscr{M}(N)+A(N)
\end{align}
where $\mathscr{S}(N)$ is a martingale and $A(N)$ is a deterministic term. Equivalently
\begin{align}
\bm{\mathsf{E}}[\mathscr{S}(N)]=\bm{\mathsf{E}}[\mathscr{M}(N)]+A(N)
\end{align}
so that
\begin{align}
A(N)=\bm{\mathsf{E}}[\mathscr{S}(N)]-\bm{\mathsf{E}}[\mathscr{M}(N)]
\end{align}
Note that $\mathscr{W}(0)=\mathscr{M}(0)$ and $\bm{\mathsf{E}}[\mathscr{M}(N)]=\mathscr{M}(0)$.
\end{lem}
\begin{prop}
The submartinagle $\mathscr{W}(N)$ for $\mathcal{F}\in(0,\mathcal{F}_{*})$ has the decomposition
\begin{align}
\mathscr{W}(N)=\mathscr{M}(N)+A(N)
\end{align}
or equivalently
\begin{align}
\bm{\mathsf{E}}[\mathscr{W}(N)]=\bm{\mathsf{E}}[\mathscr{M}(N)]+A(N)
\end{align}
where $\mathscr{M}(N)$ is a martingale given by
\begin{align}
\mathscr{M}(N)=\mathscr{W}(N)(1+\mathcal{F}(2p-1))^{-N}
\end{align}
The deterministic term $A(N)$ is
\begin{align}
A(N)=\mathscr{W}(0)(1+\mathcal{F}(2p-1))^{N}-\mathscr{M}(0)
\end{align}
\end{prop}
\begin{proof}
First, prove that $\mathscr{M}(N)$ is a martingale, which is the case if
\begin{align}
\bm{\mathsf{E}}[\mathscr{M}(N+1)|\mathcal{F}_{N })=\mathscr{M}(N)
\end{align}
where $\mathscr{M}(N)$ is an adapted process with respect to the filtration $\mathcal{F}_{N }$. Then
\begin{align}
&\bm{\mathsf{E}}[\mathscr{M}(N+1)|\mathcal{F}_{N }]=\bm{\mathsf{E}}[\mathscr{W}(N+1)(1+\mathcal{F}(2p-1))^{-N-1}|\mathfrak{F}_{N}]\nonumber\\&
=\bm{\mathsf{E}}[\mathscr{W}(N)(1+\mathcal{F}\mathscr{Z}(N+1))|\mathfrak{F}_{N}](1+\mathcal{F}(2p-1))^{-N-1}\nonumber\\&
=\mathscr{W}(N)\bm{\mathsf{E}}[(1+\mathcal{F}\mathscr{Z}(N+1)|\mathfrak{F}_{N})](1+\mathcal{F}(2p-1))^{-N-1}\nonumber\\&
=\mathscr{W}(N)[p(1+\mathcal{F})+(1-p)(1+\mathcal{F})](1+\mathcal{F}(2p-1))^{-1}(1+\mathcal{F}(2p-1))^{-N}\nonumber\\&
=\mathscr{W}(N)(1+\mathcal{F}(2p-1))(1+\mathcal{F}(2p-1))^{-1}(1+\mathcal{F}(2p-1))^{-N}\nonumber\\&
=\mathscr{W}(N)(1+\mathcal{F}(2p-1))^{-N}=\mathscr{M}(N)
\end{align}
and so $\mathscr{M}(N)$ is a martingale. Note that at the $(N+1)^{th}$ trial $\mathscr{W}(N)$ is now already known so the expectation is not required to be taken. To show that
\begin{align}
\bm{\mathsf{E}}[\mathscr{W}(N)]=\bm{\mathsf{E}}[\mathscr{M}(N)]+A(N)
\end{align}
This becomes
\begin{align}
&\mathscr{W}(0)(1+\mathcal{F}(2p+1))^{N}=\bm{\mathsf{E}}[\mathscr{M}(N)]+A(N)\nonumber\\&=\mathscr{M}(0)+\mathscr{W}(0)(1+\mathcal{F}(2p-1))^{N}-\mathscr{M}(0)\nonumber\\&
=\mathscr{W}(0)(1+\mathcal{F}(2p-1))^{N}
\end{align}
\end{proof}
\section{Volatility}
The variance and volatility of $W=U-V$, the net number of wins, were given in Lemma (2.3). Here, we make (linear) estimates for the volatility of
$\mathscr{W}(N)$, namely $\bm{\mathsf{VAR}}(\mathscr{W}(N))$ and$\bm{\sigma}(\mathscr{W}(N))$, which should hold when p is close to $\tfrac{1}{2}$ and $p>\tfrac{1}{2}$ and $N$ is large.
\begin{lem}
The accumulated wealth/money at the Nth trial or bet $\mathscr{W}(N )$ can be related to the number of wins $W=U-V$ as
\begin{align}
\mathscr{W}(N)\sim \mathscr{W}(0)\left(1+\mathcal{F}(U-V)+\mathcal{F}^{2}\left(\frac{1}{2}U(U-1)-\frac{1}{2}V(V-1)\right)+\mathscr{Z}(\mathcal{F}^{3}\right)
\end{align}
with the linear first-order approximation
\begin{align}
\mathscr{W}(N)\sim \mathscr{W}(0)(1+\mathcal{F}(U-V))
\end{align}
\end{lem}
\begin{proof}
\begin{align}
\mathscr{W}(N)\sim \mathscr{W}(0)(1+\mathcal{F})^{U}(1-\mathcal{F})^{V}
\end{align}
The terms $(1+\mathcal{F})^{U}$ and $(1-\mathcal{F})^{V}$ can be binomially expanded to 2nd order as
\begin{align}
&(1+\mathcal{F})^{U}=\sum_{Q=0}^{\infty}\binom{U}{Q}\mathcal{F}^{Q}\sim 1+U\mathcal{F}+\frac{1}{2}U(U-1)\mathcal{F}^{2}\\&
(1-\mathcal{F})^{V}=\sum_{Q=0}^{\infty}(-1)^{Q}\binom{V}{Q}\mathcal{F}^{Q}\sim 1+V\mathcal{F}-\frac{1}{2}V(V-1)\mathcal{F}^{2}
\end{align}
Then
\begin{align}
\mathscr{W}(N)=\mathscr{W}(0)(1+U\mathcal{F}+\frac{1}{2}U(U-1)\mathcal{F}^{2})(1-V\mathcal{F}+\frac{1}{2}V(V-1)\mathcal{F}^{2})+\mathscr{Z}(\mathcal{F}^{3})
\end{align}
which is
\begin{align}
\mathscr{W}(N)=\mathscr{W}(0)(1+\mathcal{F}(U-V)+\frac{1}{2}U(U-1)\mathcal{F}^{2}-\frac{1}{2}V(V-1)\mathcal{F}^{2}
\end{align}
and to a 1st-order linear approximation
\begin{align}
\mathscr{W}(N)=\mathscr{W}(0)(1+\mathcal{F}(U-V))
\end{align}
This also follows from
\begin{align}
\bm{\mathsf{E}}[\mathscr{W}(N)]&=\mathscr{W}(0)(1+p\mathcal{F})^{N}(1-q\mathcal{F})^{N})]=\mathscr{W}(0)(1+NP\mathcal{F})(1-Np\mathcal{F})+\mathscr{Z}(\mathcal{F}^{3})\nonumber\\&
=\mathscr{W}(0)(1+N\mathcal{F}(p-q))=\mathscr{W}(0)(1+\mathcal{F}(N p-Nq))\nonumber\\&
=\mathscr{W}(0)(1+\mathcal{F}\bm{\mathsf{E}}[U]-\bm{\mathsf{E}}[V])
\end{align}
Hence since $\bm{\mathsf{E}}[U]=pN$ and $\bm{\mathsf{E}}[V]=qN$
\begin{align}
\mathscr{W}(N)=\mathscr{W}(0)(1+\mathcal{F}(U-V))
\end{align}
\end{proof}
\begin{prop}
The variance can be estimated to a first-order (linear)approximation as
\begin{align}
\bm{\mathsf{VAR}}(\mathscr{W}(N))\sim 2|\mathscr{W}(0)|^{2}Np(1-p)
\end{align}
and the volatility follows as
\begin{align}
\mathbf{\sigma}(\mathscr{W}(N))\sim \sqrt{2|\mathscr{W}(0)|^{2}Np(1-p)}
\end{align}
\end{prop}
\begin{proof}
The variance is
\begin{align}
&\bm{\mathsf{VAR}}(\mathscr{W}(N))\sim \bm{\mathsf{VAR}}(\mathscr{W}(0)+\mathcal{F}\mathscr{W}(0)(U-V))\nonumber\\&=\bm{\mathsf{VAR}}(\mathcal{F}\mathscr{W}(0)(U-V))=\mathcal{F}^{2}|\mathscr{W}(0)|^{2}\bm{\mathsf{VAR}}(U-V)
\end{align}
since $\bm{\mathsf{VAR}}(C)=0$ for any constant $C$, and $\bm{\mathsf{VAR}}(\mathcal{F} X)=\mathcal{F}^{2}\bm{\mathsf{VAR}}(X)$ for a random variable $X$. From (-) it is
already known that $\bm{\mathsf{VAR}}(U-V)=2Np(1-p)$ and so the result follows
\begin{align}
\bm{\mathsf{VAR}}(\mathscr{W}(N))\sim 2|\mathscr{W}(0)|^{2}Np(1-p)\mathcal{F}^{2}
\end{align}
\end{proof}
\begin{cor}
If $\mathcal{F}=\mathcal{F}_{K}=2p-1$, the Kelly fraction, then the variance becomes
\begin{align}
\bm{\mathsf{VAR}}(\mathscr{W}(N))\sim 2|\mathscr{W}(0)|^{2}Np(1-p)(2p-1)^{2}
\end{align}
Expanding
\begin{align}
\bm{\mathsf{VAR}}(\mathscr{W}(N))\sim 2|\mathscr{W}(0)|^{2}Np(1-p)(4p^{2}-4p+1)
\end{align}
Since $p$ is small and N is large, the nonlinear terms higher than quadratic can be dropped to give the linear approximation
\begin{align}
\bm{\mathsf{VAR}}(\mathscr{W}(N))\sim 2|\mathscr{W}(0)|^{2}Np(1-p)
\end{align}
and the volatility is
\begin{align}
\mathbf{\sigma}(\mathscr{W}(N))\sim \sqrt{2|\mathscr{W}(0)|^{2}Np(1-p)}
\end{align}
\end{cor}
Generally, one has at most $p=0.51$ or $p=0.52$; that is, the edge $\mathsf{E}=p-q>0$ will be very small so this linear approximation holds.
\section{Fractional Kelly fractions}
Since $p>q$ there an edge in the players favour so that $p>q$ and $\mathcal{F}_{K}=p-q=2p-1$. If $p\le q=\tfrac{1}{2}$ then $\mathcal{F}_{K}\le 0$ and one does not play the game. However, in a real gambling or financial scenario it is possible that one may have overestimated the edge or advantage. It is therefore safer to take some fraction $f<1$ of the Kelly fraction $\mathcal{F}_{K}$. The down side is that the growth rate is slower since this will correspond to a point on the utility function graph left of the peak or critical point. However, the advantage is that the fluctuations are now reduced and that the volatility is reduced.
\begin{prop}
Suppose the estimated edge is $\mathcal{E}=p-q>0$ and $\mathcal{F}_{K}=p-q=2p-1$. Let $f\in[\frac{1}{2},1)$ then define a 'fractional Kelly fraction' as
\begin{align}
F_{K}=f\mathcal{F}_{K}
\end{align}
\end{prop}
For example, choosing $f=\tfrac{2}{3}$ then $F_{K}=\frac{2}{3}\mathcal{F}_{K}$. Suppose $p=0.515$ and $q=1-p=0.475$ then $\mathcal{F}_{K}=0.03$. If one chooses
$f=\tfrac{2}{3}$ then $F_{K}=f\mathcal{F}_{K}=\tfrac{2}{3}\tfrac{3}{100}=\tfrac{2}{100}=\tfrac{1}{50}$. Hence if the wealth/liquity at the $(N-1)^{th}$ bet or wager is
$\mathscr{W}(N-1)$ the at the $N^{th}$ bet one would wager an amount
\begin{align}
B(N)=F_{K}\mathscr{W}(N-1)=f\mathcal{F}_{K}\mathscr{W}(N-1)=\tfrac{1}{50}\mathscr{W}(N)
\end{align}
The value of $f$ of course is arbitrary and one could choose $f=\tfrac{1}{2}$ or $f=\tfrac{3}{4}$ and so on, but of course, the smaller the fraction $f$, the slower the rate of growth of $\mathscr{W}(N)$ with $N$.

Figure 5 and 6 plot the expectation $\bm{\mathsf{E}}[\mathscr{W}(N)] vs N$, and the volatility $\sigma(\mathscr{W}(N)) vs N$ for $p=0.52$ and $q=1-p=0.48$ for the Kelly fraction $\mathcal{F}_{K}=p-q=2p-1=0.04=\tfrac{1}{25}$, and also for the fractional Kelly fraction $F_{K}=f\mathcal{F}_{K}=\tfrac{2}{3}\mathcal{F}_{K}=\tfrac{2}{3}\tfrac{1}{25}=\tfrac{2}{75}$. Here $\mathscr{W}(0)=1000$.
\begin{figure}
\begin{center}
\includegraphics[height=3.0in,width=6.0in]{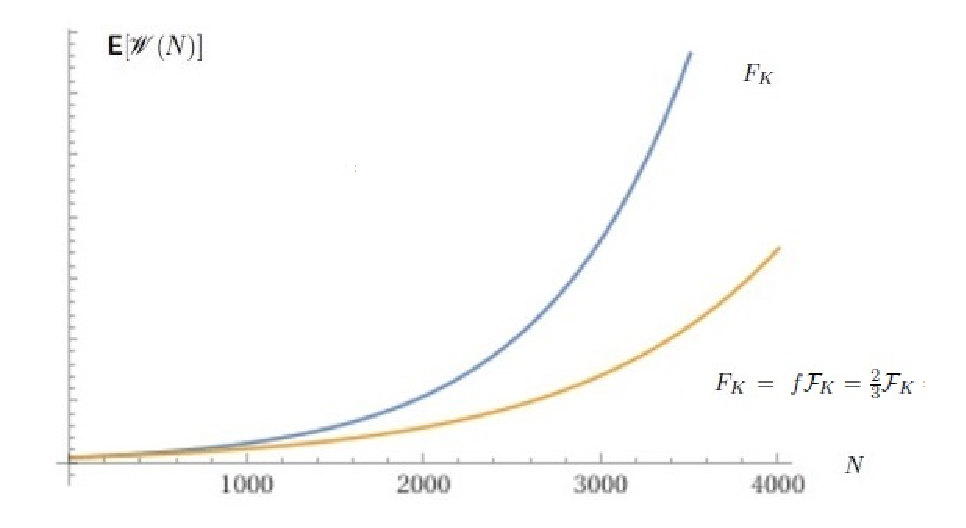}
\caption{Plots of $\bm{\mathsf{E}}[\mathscr{W}(N)]$ vs N  for $p=0.52$ and $\mathcal{F}_{K}=0.04$ and $F_{K}=f\mathcal{F}_{K}=\tfrac{2}{3}\mathcal{F}_{K}=\tfrac{2}{75}$}
\end{center}
\end{figure}
\begin{figure}
\begin{center}
\includegraphics[height=2.45in,width=6.0in]{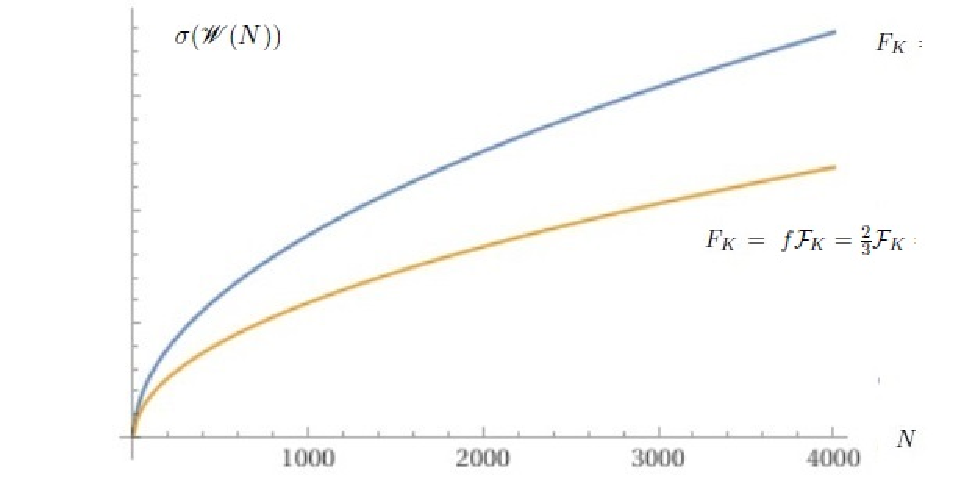}
\caption{Plots of volatility $\bm{\sigma}(N)$ vs N  for $p=0.52$ and $\mathcal{F}_{K}=0.04$ and $F_{K}=f\mathcal{F}_{K}=\tfrac{2}{3}\mathcal{F}_{K}=\tfrac{2}{75}$}
\end{center}
\end{figure}
In Figure 5, one can see that the expected value $\bm{\mathsf{E}}[\mathscr{W}(N)]$ for any N, decreases for the fraction $F_{K}=\tfrac{2}{3}\mathcal{F}_{K}$ for the same probability $p=0.52$ so as expected the growth rate is now less than the optimum growth rate that occurs for $\mathcal{F}_{K}$. The trade off, as shown in Figure 6, is that the volatility is now decreased for $f=\tfrac{2}{3}$.
\clearpage
\appendix
\renewcommand{\theequation}{\Alph{section}.\arabic{equation}}
\section{The moment generating function for binomial distributions}
In this section, expressions for variance and volatility of the liquidity $\mathscr{W}(N)$ are derived for arbitrary F and for the Kelly fraction $\mathcal{F}_{K}$.
Also, the expectations $\bm{\mathsf{E}}[\mathscr{W}(N)]$ and $\bm{\mathsf{E}}[|\mathscr{W}(N)|^{2}]$. The following preliminary definition and lemma are first required.
\begin{defn}
The moment generating function (MGF) of a random variable $X$ is formally defined as
\begin{align}
\bm{\mathsf{M}}_{N}(X,\xi)=\bm{\mathsf{E}}[\exp(\xi X)]
\end{align}
with $\xi>0$. For the binomial random variables $(U,V)$ with $U\sim Binom(N,P)$ and $V\sim Binom (N,Q)$, the MGFs are
\begin{align}
&\bm{\mathsf{M}}_{N}(U,\xi)=\bm{\mathsf{E}}[\exp(\xi U)]=\sum_{\alpha=0}^{N}\exp(\xi \alpha)\bm{\mathsf{P}}(U=\alpha;N)=\sum_{\alpha=0}^{N}\exp(\xi \alpha)\binom{N}{\alpha}\mathit{p}^{\alpha}\mathit{q}^{N-\alpha}\\&
\bm{\mathsf{M}}_{N}(V,\xi)=\bm{\mathsf{E}}[\exp(\xi V)]=\sum_{\beta=0}^{N}\exp(\xi \beta)\bm{\mathsf{P}}(V=\beta;N)=\sum_{\beta=0}^{N}\exp(\xi \beta)\binom{N}{\beta}\mathit{q}^{\beta}\mathit{q}^{N-\beta}
\end{align}
\end{defn}
\begin{lem}
For the binomial random variables $(U,V)$ with $U\sim Binom(N,P)$ and $V\sim Binom (N,Q)$, the MGFs are
\begin{align}
&\bm{\mathsf{M}}_{N}(U,\xi)=\bm{\mathsf{E}}[\exp(\xi U)]=(1-p+p\exp(\xi))^{N}\\&
\bm{\mathsf{M}}_{N}(V,\xi)=\bm{\mathsf{E}}[\exp(\xi U)]=(1-q+q\exp(\xi))^{N}
\end{align}
\end{lem}
\begin{proof}
\begin{align}
\bm{\mathsf{M}}_{N}(U,\xi)&=\bm{\mathsf{E}}[\exp(\xi U)]=\sum_{\alpha=0}^{N}\exp(\xi U)\bm{\mathsf{P}}(U=\alpha;N)\nonumber\\&=\sum_{\alpha=0}^{N}\exp(\xi \alpha)\binom{N}{\alpha}\mathit{p}^{\alpha}\mathit{q}^{N-\alpha}\equiv \sum_{\alpha=0}^{N}\exp(\xi \alpha\binom{N}{\alpha}\mathit{p}^{\alpha}(1-\mathit{p})^{N-\alpha}\nonumber\\&
=(1-\mathit{p})^{N}\sum_{\alpha=0}^{N}\binom{N}{\alpha}\frac{\mathit{p}^{\alpha}\exp(\xi \alpha)}{(1-\mathit{p})^{\alpha}}\nonumber\\&
=(1-\mathit{p})^{N}\sum_{\alpha=0}^{N}\binom{N}{\alpha}\left|\frac{p\exp(\xi)}{(1-\mathit{p})}\right|^{\alpha}\nonumber\\&
=(1-\mathit{p})^{N}\sum_{\alpha=0}^{N}\binom{N}{\alpha}\mathscr{Z}^{\alpha}
\end{align}
where $\mathscr{Z}=p\exp(\xi)/(1-\mathit{p})$. Now applying the Binomial Theorem
\begin{align}
\sum_{\alpha=0}^{N}\binom{N}{\alpha}\mathscr{Z}^{\alpha}=(1+\mathscr{Z})^{N}
\end{align}
so that now
\begin{align}
&\bm{\mathsf{M}}_{N}(U,\xi)=(1-\mathit{p})^{N}\sum_{\alpha=0}^{N}\exp(\xi U)\binom{N}{\alpha}\left|\frac{p\exp(\xi)}{(1-\mathit{p})}\right|^{\alpha}=(1-\mathit{p})^{N}\left|1+ \frac{p\exp(\xi)}{1-p}\right|^{N}\nonumber\\&
=(1-\mathit{p})^{N}\left|\frac{1-p+p\exp(\xi)}{1-p}\right|^{N}=\frac{(1-\mathit{p})^{N}}{(1-\mathit{p})^{N}}(1-p+p\exp(\xi))^{N}\nonumber\\&
=(1-p+p\exp(\xi))^{N}
\end{align}
and similarly for $\bm{\mathsf{M}}_{N}(V,\xi)$.
\end{proof}

\end{document}